\documentclass[11pt,twoside,a4paper]{article}
\usepackage{color,amscd,amsmath,amssymb,amsthm,latexsym,stmaryrd,mathdots,authblk,mathabx,shuffle,hyperref,tikz,tkz-tab} 
\usepackage[latin1]{inputenc}
\usepackage[T1]{fontenc}   
\usepackage[english]{babel}
\providecommand{\keywords}[1]{\textbf{\textit{Keywords.}} #1}
\providecommand{\AMSclass}[1]{\textbf{\textit{AMS classification.}} #1}

\input{xy}
\xyoption{all}
\newcommand{\rond}[1]{*++[o][F-]{#1}}

\newcommand{\tun}{\begin{tikzpicture}[line cap=round,line join=round,>=triangle 45,x=0.5cm,y=0.5cm]
\clip(-0.2,-0.1) rectangle (0.2,0.2);
\begin{scriptsize}
\draw [fill=black] (0.,0.) circle (1pt);
\end{scriptsize}
\end{tikzpicture}}

\newcommand{\tdeux}{\begin{tikzpicture}[line cap=round,line join=round,>=triangle 45,x=0.5cm,y=0.5cm]
\clip(-.2,-.1) rectangle (0.2,0.7);
\draw [line width=.5pt] (0.,0.5)-- (0.,0.);
\begin{scriptsize}
\draw [fill=black] (0.,0.) circle (1pt);
\draw [fill=black] (0.,0.5) circle (1pt);
\end{scriptsize}
\end{tikzpicture}}

\newcommand{\ttroisun}{\begin{tikzpicture}[line cap=round,line join=round,>=triangle 45,x=0.5cm,y=0.5cm]
\clip(-0.5,-0.1) rectangle (0.5,0.7);
\draw [line width=0.5pt] (0.,0.)-- (-0.3,0.5);
\draw [line width=0.5pt] (0.,0.)-- (0.3,0.5);
\begin{scriptsize}
\draw [fill=black] (-0.3,0.5) circle (1pt);
\draw [fill=black] (0.,0.) circle (1pt);
\draw [fill=black] (0.3,0.5) circle (1pt);
\end{scriptsize}
\end{tikzpicture}}
\newcommand{\ttroisdeux}{\begin{tikzpicture}[line cap=round,line join=round,>=triangle 45,x=0.5cm,y=0.5cm]
\clip(-.2,-.1) rectangle (0.2,1.2);
\draw [line width=0.5pt] (0.,0.5)-- (0.,0.);
\draw [line width=0.5pt] (0.,0.5)-- (0.,1.);
\begin{scriptsize}
\draw [fill=black] (0.,0.) circle (1pt);
\draw [fill=black] (0.,0.5) circle (1pt);
\draw [fill=black] (0.,1.) circle (1pt);
\end{scriptsize}
\end{tikzpicture}}

\newcommand{\tquatreun}{\begin{tikzpicture}[line cap=round,line join=round,>=triangle 45,x=0.5cm,y=0.5cm]
\clip(-0.5,-0.1) rectangle (0.5,0.7);
\draw [line width=0.5pt] (0.,0.)-- (-0.3,0.5);
\draw [line width=0.5pt] (0.,0.)-- (0.3,0.5);
\draw [line width=0.5pt] (0.,0.)-- (0.,0.5);
\begin{scriptsize}
\draw [fill=black] (-0.3,0.5) circle (1.0pt);
\draw [fill=black] (0.,0.) circle (1.0pt);
\draw [fill=black] (0.3,0.5) circle (1.0pt);
\draw [fill=black] (0.,0.5) circle (1.0pt);
\end{scriptsize}
\end{tikzpicture}}
\newcommand{\tquatredeux}{\begin{tikzpicture}[line cap=round,line join=round,>=triangle 45,x=0.5cm,y=0.5cm]
\clip(-0.5,-0.1) rectangle (0.5,1.2);
\draw [line width=0.5pt] (0.,0.)-- (-0.3,0.5);
\draw [line width=0.5pt] (0.,0.)-- (0.3,0.5);
\draw [line width=0.5pt] (-0.3,0.5)-- (-0.3,1.);
\begin{scriptsize}
\draw [fill=black] (-0.3,0.5) circle (1.0pt);
\draw [fill=black] (0.,0.) circle (1.0pt);
\draw [fill=black] (0.3,0.5) circle (1.0pt);
\draw [fill=black] (-0.3,1.) circle (1.0pt);
\end{scriptsize}
\end{tikzpicture}}
\newcommand{\tquatretrois}{\begin{tikzpicture}[line cap=round,line join=round,>=triangle 45,x=0.5cm,y=0.5cm]
\clip(-0.5,-0.1) rectangle (0.5,1.2);
\draw [line width=0.5pt] (0.,0.)-- (-0.3,0.5);
\draw [line width=0.5pt] (0.,0.)-- (0.3,0.5);
\draw [line width=0.5pt] (0.3,0.5)-- (0.3,1.);
\begin{scriptsize}
\draw [fill=black] (-0.3,0.5) circle (1.0pt);
\draw [fill=black] (0.,0.) circle (1.0pt);
\draw [fill=black] (0.3,0.5) circle (1.0pt);
\draw [fill=black] (0.3,1.) circle (1.0pt);
\end{scriptsize}
\end{tikzpicture}}
\newcommand{\tquatrequatre}{\begin{tikzpicture}[line cap=round,line join=round,>=triangle 45,x=0.5cm,y=0.5cm]
\clip(-0.5,-0.1) rectangle (0.5,1.7);
\draw [line width=0.5pt] (0.,0.)-- (0.,0.5);
\draw [line width=0.5pt] (0.,0.5)-- (0.3,1.);
\draw [line width=0.5pt] (0.,0.5)-- (-0.3,1.);
\begin{scriptsize}
\draw [fill=black] (0.,0.) circle (1.0pt);
\draw [fill=black] (0.,0.5) circle (1.0pt);
\draw [fill=black] (-0.3,1.) circle (1.0pt);
\draw [fill=black] (0.3,1.) circle (1.0pt);
\end{scriptsize}
\end{tikzpicture}}
\newcommand{\tquatrecinq}{\begin{tikzpicture}[line cap=round,line join=round,>=triangle 45,x=0.5cm,y=0.5cm]
\clip(-.2,-.1) rectangle (0.5,1.7);
\draw [line width=0.5pt] (0.,0.)-- (0.,0.5);
\draw [line width=0.5pt] (0.,0.5)-- (0.,1.);
\draw [line width=0.5pt] (0.,1.)-- (0.,1.5);
\begin{scriptsize}
\draw [fill=black] (0.,0.) circle (1.0pt);
\draw [fill=black] (0.,0.5) circle (1.0pt);
\draw [fill=black] (0.,1.) circle (1.0pt);
\draw [fill=black] (0.,1.5) circle (1.0pt);
\end{scriptsize}
\end{tikzpicture}}

\newcommand{\tdun}[1]{\begin{tikzpicture}[line cap=round,line join=round,>=triangle 45,x=0.5cm,y=0.5cm]
\clip(-.2,-.1) rectangle (0.5,0.5);
\begin{scriptsize}
\draw [fill=black] (0.,0.) circle (1pt);
\end{scriptsize}
\draw(0.3,0.1) node {\tiny #1};
\end{tikzpicture}}

\newcommand{\tddeux}[2]{\begin{tikzpicture}[line cap=round,line join=round,>=triangle 45,x=0.5cm,y=0.5cm]
\clip(-.2,-.1) rectangle (0.5,1.);
\draw [line width=.5pt] (0.,0.5)-- (0.,0.);
\begin{scriptsize}
\draw [fill=black] (0.,0.) circle (1pt);
\draw [fill=black] (0.,0.5) circle (1pt);
\end{scriptsize}
\draw(0.3,0.1) node {\tiny #1};
\draw(0.3,0.6) node {\tiny #2};
\end{tikzpicture}}

\newcommand{\tdtroisun}[3]{\begin{tikzpicture}[line cap=round,line join=round,>=triangle 45,x=0.5cm,y=0.5cm]
\clip(-0.5,-0.1) rectangle (0.5,1.2);
\draw [line width=0.5pt] (0.,0.)-- (-0.3,0.5);
\draw [line width=0.5pt] (0.,0.)-- (0.3,0.5);
\begin{scriptsize}
\draw [fill=black] (-0.3,0.5) circle (1pt);
\draw [fill=black] (0.,0.) circle (1pt);
\draw [fill=black] (0.3,0.5) circle (1pt);
\end{scriptsize}
\draw(0.35,0.1) node {\tiny #1};
\draw(0.3,0.8) node {\tiny #2};
\draw(-0.3,0.8) node {\tiny #3};
\end{tikzpicture}}
\newcommand{\tdtroisdeux}[3]{\begin{tikzpicture}[line cap=round,line join=round,>=triangle 45,x=0.5cm,y=0.5cm]
\clip(-.2,-.1) rectangle (0.5,1.5);
\draw [line width=0.5pt] (0.,0.5)-- (0.,0.);
\draw [line width=0.5pt] (0.,0.5)-- (0.,1.);
\begin{scriptsize}
\draw [fill=black] (0.,0.) circle (1pt);
\draw [fill=black] (0.,0.5) circle (1pt);
\draw [fill=black] (0.,1.) circle (1pt);
\end{scriptsize}
\draw(0.3,0.1) node {\tiny #1};
\draw(0.3,0.6) node {\tiny #2};
\draw(0.3,1.1) node {\tiny #3};
\end{tikzpicture}}

\newcommand{\tdquatreun}[4]{\begin{tikzpicture}[line cap=round,line join=round,>=triangle 45,x=0.5cm,y=0.5cm]
\clip(-0.5,-0.1) rectangle (0.5,1.2);
\draw [line width=0.5pt] (0.,0.)-- (-0.3,0.5);
\draw [line width=0.5pt] (0.,0.)-- (0.3,0.5);
\draw [line width=0.5pt] (0.,0.)-- (0.,0.5);
\begin{scriptsize}
\draw [fill=black] (-0.3,0.5) circle (1.0pt);
\draw [fill=black] (0.,0.) circle (1.0pt);
\draw [fill=black] (0.3,0.5) circle (1.0pt);
\draw [fill=black] (0.,0.5) circle (1.0pt);
\end{scriptsize}
\draw(0.35,0.1) node {\tiny #1};
\draw(0.3,0.8) node {\tiny #2};
\draw(0.,0.8) node {\tiny #3};
\draw(-0.3,0.8) node {\tiny #4};
\end{tikzpicture}}

\newcommand{\tdquatrequatre}[4]{\begin{tikzpicture}[line cap=round,line join=round,>=triangle 45,x=0.5cm,y=0.5cm]
\clip(-0.5,-0.1) rectangle (0.5,1.7);
\draw [line width=0.5pt] (0.,0.)-- (0.,0.5);
\draw [line width=0.5pt] (0.,0.5)-- (0.3,1.);
\draw [line width=0.5pt] (0.,0.5)-- (-0.3,1.);
\begin{scriptsize}
\draw [fill=black] (0.,0.) circle (1.0pt);
\draw [fill=black] (0.,0.5) circle (1.0pt);
\draw [fill=black] (-0.3,1.) circle (1.0pt);
\draw [fill=black] (0.3,1.) circle (1.0pt);
\end{scriptsize}
\draw(0.3,0.1) node {\tiny #1};
\draw(0.3,0.5) node {\tiny #2};
\draw(0.3,1.3) node {\tiny #3};
\draw(-0.3,1.3) node {\tiny #4};
\end{tikzpicture}}

\newcommand{\ptroisun}{\begin{tikzpicture}[line cap=round,line join=round,>=triangle 45,x=0.5cm,y=0.5cm]
\clip(-0.5,-0.1) rectangle (0.5,0.7);
\draw [line width=0.5pt] (0.,0.5)-- (-0.3,0.);
\draw [line width=0.5pt] (0.,0.5)-- (0.3,0.);
\begin{scriptsize}
\draw [fill=black] (-0.3,0.) circle (1pt);
\draw [fill=black] (0.,0.5) circle (1pt);
\draw [fill=black] (0.3,0.) circle (1pt);
\end{scriptsize}
\end{tikzpicture}}

\newcommand{\pquatreun}{\begin{tikzpicture}[line cap=round,line join=round,>=triangle 45,x=0.5cm,y=0.5cm]
\clip(-0.5,-0.1) rectangle (0.5,0.7);
\draw [line width=0.5pt] (0.,0.5)-- (-0.3,0.);
\draw [line width=0.5pt] (0.,0.5)-- (0.3,0.);
\draw [line width=0.5pt] (0.,0.5)-- (0.,0.);
\begin{scriptsize}
\draw [fill=black] (-0.3,0.) circle (1.0pt);
\draw [fill=black] (0.,0.5) circle (1.0pt);
\draw [fill=black] (0.3,0.) circle (1.0pt);
\draw [fill=black] (0.,0.) circle (1.0pt);
\end{scriptsize}
\end{tikzpicture}}
\newcommand{\pquatredeux}{\begin{tikzpicture}[line cap=round,line join=round,>=triangle 45,x=0.5cm,y=0.5cm]
\clip(-0.5,-0.1) rectangle (0.5,1.2);
\draw [line width=0.5pt] (0.,1.)-- (-0.3,0.5);
\draw [line width=0.5pt] (0.,1.)-- (0.3,0.5);
\draw [line width=0.5pt] (-0.3,0.5)-- (-0.3,0.);
\begin{scriptsize}
\draw [fill=black] (-0.3,0.5) circle (1.0pt);
\draw [fill=black] (0.,1.) circle (1.0pt);
\draw [fill=black] (0.3,0.5) circle (1.0pt);
\draw [fill=black] (-0.3,0.) circle (1.0pt);
\end{scriptsize}
\end{tikzpicture}}
\newcommand{\pquatretrois}{\begin{tikzpicture}[line cap=round,line join=round,>=triangle 45,x=0.5cm,y=0.5cm]
\clip(-0.5,-0.1) rectangle (0.5,1.2);
\draw [line width=0.5pt] (0.,1.)-- (-0.3,0.5);
\draw [line width=0.5pt] (0.,1.)-- (0.3,0.5);
\draw [line width=0.5pt] (0.3,0.5)-- (0.3,0.);
\begin{scriptsize}
\draw [fill=black] (-0.3,0.5) circle (1.0pt);
\draw [fill=black] (0.,1.) circle (1.0pt);
\draw [fill=black] (0.3,0.5) circle (1.0pt);
\draw [fill=black] (0.3,0.) circle (1.0pt);
\end{scriptsize}
\end{tikzpicture}}
\newcommand{\pquatrequatre}{\begin{tikzpicture}[line cap=round,line join=round,>=triangle 45,x=0.5cm,y=0.5cm]
\clip(-0.5,-0.1) rectangle (0.5,1.7);
\draw [line width=0.5pt] (0.,1.)-- (0.,0.5);
\draw [line width=0.5pt] (0.,0.5)-- (0.3,0.);
\draw [line width=0.5pt] (0.,0.5)-- (-0.3,0.);
\begin{scriptsize}
\draw [fill=black] (0.,1.) circle (1.0pt);
\draw [fill=black] (0.,0.5) circle (1.0pt);
\draw [fill=black] (-0.3,0.) circle (1.0pt);
\draw [fill=black] (0.3,0.) circle (1.0pt);
\end{scriptsize}
\end{tikzpicture}}
\newcommand{\pquatrecinq}{\begin{tikzpicture}[line cap=round,line join=round,>=triangle 45,x=0.5cm,y=0.5cm]
\clip(-0.5,-0.1) rectangle (0.5,0.7);
\draw [line width=0.5pt] (-0.3,0.)-- (-0.3,0.5);
\draw [line width=0.5pt] (0.3,0.)-- (0.3,0.5);
\draw [line width=0.5pt] (-0.3,0.)-- (0.3,0.5);
\begin{scriptsize}
\draw [fill=black] (-0.3,0.) circle (1pt);
\draw [fill=black] (0.3,0.) circle (1pt);
\draw [fill=black] (-0.3,0.5) circle (1pt);
\draw [fill=black] (0.3,0.5) circle (1pt);
\end{scriptsize}
\end{tikzpicture}}
\newcommand{\pquatresix}{\begin{tikzpicture}[line cap=round,line join=round,>=triangle 45,x=0.5cm,y=0.5cm]
\clip(-0.5,-0.1) rectangle (0.5,0.7);
\draw [line width=0.5pt] (-0.3,0.)-- (-0.3,0.5);
\draw [line width=0.5pt] (0.3,0.)-- (0.3,0.5);
\draw [line width=0.5pt] (0.3,0.)-- (-0.3,0.5);
\begin{scriptsize}
\draw [fill=black] (-0.3,0.) circle (1pt);
\draw [fill=black] (0.3,0.) circle (1pt);
\draw [fill=black] (-0.3,0.5) circle (1pt);
\draw [fill=black] (0.3,0.5) circle (1pt);
\end{scriptsize}
\end{tikzpicture}}
\newcommand{\pquatresept}{\begin{tikzpicture}[line cap=round,line join=round,>=triangle 45,x=0.5cm,y=0.5cm]
\clip(-0.5,-0.1) rectangle (0.5,0.7);
\draw [line width=0.5pt] (-0.3,0.)-- (-0.3,0.5);
\draw [line width=0.5pt] (0.3,0.)-- (0.3,0.5);
\draw [line width=0.5pt] (0.3,0.)-- (-0.3,0.5);
\draw [line width=0.5pt] (-0.3,0.)-- (0.3,0.5);
\begin{scriptsize}
\draw [fill=black] (-0.3,0.) circle (1pt);
\draw [fill=black] (0.3,0.) circle (1pt);
\draw [fill=black] (-0.3,0.5) circle (1pt);
\draw [fill=black] (0.3,0.5) circle (1pt);
\end{scriptsize}
\end{tikzpicture}}
\newcommand{\pquatrehuit}{\begin{tikzpicture}[line cap=round,line join=round,>=triangle 45,x=0.5cm,y=0.5cm]
\clip(-0.5,-0.1) rectangle (0.5,1.2);
\draw [line width=0.5pt] (0.,1.)-- (-0.3,0.5);
\draw [line width=0.5pt] (0.,1.)-- (0.3,0.5);
\draw [line width=0.5pt] (-0.3,0.5)-- (0.,0.);
\draw [line width=0.5pt] (0.3,0.5)-- (0.,0.);
\begin{scriptsize}
\draw [fill=black] (-0.3,0.5) circle (1.0pt);
\draw [fill=black] (0.,1.) circle (1.0pt);
\draw [fill=black] (0.3,0.5) circle (1.0pt);
\draw [fill=black] (0.,0.) circle (1.0pt);
\end{scriptsize}
\end{tikzpicture}}

\newcommand{\pdtroisun}[3]{\begin{tikzpicture}[line cap=round,line join=round,>=triangle 45,x=0.5cm,y=0.5cm]
\clip(-0.7,-0.1) rectangle (0.7,1.2);
\draw [line width=0.5pt] (0.,0.5)-- (-0.3,0.);
\draw [line width=0.5pt] (0.,0.5)-- (0.3,0.);
\begin{scriptsize}
\draw [fill=black] (-0.3,0.) circle (1pt);
\draw [fill=black] (0.,0.5) circle (1pt);
\draw [fill=black] (0.3,0.) circle (1pt);
\end{scriptsize}
\draw(0.,0.8) node {\tiny #1};
\draw(-0.6,0.1) node {\tiny #2};
\draw(0.6,0.1) node {\tiny #3};
\end{tikzpicture}}

\definecolor{vert}{rgb}{0.,0.5,0.}

\title{Plane posets, special posets, and permutations}
\date{}
\author{Lo\"\i c Foissy}
\affil{\small{Univ. Littoral Côte d'Opale, UR 2597
LMPA, Laboratoire de Mathématiques Pures et Appliquées Joseph Liouville
F-62100 Calais, France}.\\ Email: \texttt{foissy@univ-littoral.fr}}

\theoremstyle{plain}
\newtheorem{theo}{Theorem}[section]
\newtheorem{lemma}[theo]{Lemma}
\newtheorem{cor}[theo]{Corollary}
\newtheorem{prop}[theo]{Proposition}
\newtheorem{defi}[theo]{Definition}

\theoremstyle{remark}
\newtheorem{remark}{Remark}[section]
\newtheorem{notation}{Notations}[section]
\newtheorem{example}{Example}[section]

\newcommand{\K}{\mathbb{K}}

\renewcommand{\geq}{\geqslant}
\renewcommand{\leq}{\leqslant}

\DeclareMathOperator{\tdelta}{\tilde{\Delta}}

\DeclareMathOperator{\Ker}{\mathrm{Ker}}

\DeclareMathOperator{\Char}{\mathrm{char}}
\DeclareMathOperator{\Card}{\mathrm{card}}

\newcommand{\DP}{\mathcal{DP}}
\newcommand{\PP}{\mathcal{PP}}
\newcommand{\WNP}{\mathcal{WNP}}
\newcommand{\PF}{\mathcal{PF}}
\newcommand{\OF}{\mathcal{OF}}
\newcommand{\HOF}{\mathcal{HOF}}
\newcommand{\HOP}{\mathcal{HOP}}
\newcommand{\SP}{\mathcal{SP}}
\newcommand{\SPP}{\mathcal{S}\PP}
\newcommand{\SWNP}{\mathcal{S}\WNP}
\newcommand{\SPF}{\mathcal{S}\PF}

\newcommand{\D}{\mathcal{D}}

\renewcommand{\P}{\mathcal{P}}
\newcommand{\h}{\mathcal{H}}

\renewcommand{\S}{\mathfrak{S}}

\newcommand{\FQSym}{\mathbf{FQSym}}
\newcommand{\PQSym}{\mathbf{PQSym}}

\newcommand{\prodg}{\rightsquigarrow}
\newcommand{\prodh}{\lightning}

\newcommand{\rondrond}[1]{*++[o][F--]{#1}}

\begin{document}

\maketitle

\begin{abstract}
We study the self-dual Hopf algebra $\h_{\SP}$ of special posets introduced by Malvenuto and Reutenauer
and the Hopf algebra morphism from $\h_{\SP}$ to the Hopf algebra of free quasi-symmetric functions $\FQSym$ given by linear extensions.
In particular, we construct two Hopf subalgebras both isomorphic to $\FQSym$; the first one is based on plane posets, the second one on heap-ordered forests. 
An explicit isomorphism between these two Hopf subalgebras is also defined, with the help of two combinatorial transformations on special posets. 
The restriction of the Hopf pairing of $\h_{\SP}$ to these Hopf subalgebras and others is also studied, as well as certain isometries between them.
These problems are solved using duplicial and dendriform structures. 

\end{abstract}

\\

\AMSclass{06A11, 05A05, 16W30, 17A30.}

\tableofcontents

\section*{Introduction}

The Hopf algebra of double posets is introduced in \cite{MR2}. Recall that a \emph{double poset} is a finite set with two partial orders;
the set of isoclasses of double posets is given a structure of monoid, with a product called \emph{composition} (definition \ref{4}). The algebra of this monoid
is given a coassociative coproduct, with the help of the notion of \emph{ideal} of a double poset. We then obtain a graded, connected Hopf algebra,
non commutative and non cocommutative. This Hopf algebra $\h_{\DP}$ is self-dual: it has a nondegenerate Hopf pairing
$\langle-,-\rangle$, such that the pairing of two double posets is given by the number of \emph{pictures} between these double posets
(definition \ref{6}); see \cite{Foissy1} for more details on the nondegeneracy of this pairing. 

Other algebraic structures are constructed on $\h_{\DP}$ in \cite{Foissy1}.  In particular, a second product is defined on $\h_{\DP}$,
making it a free $2$-$As$ Hopf algebra \cite{LR1}. As a consequence, this object is closely related to operads and the theory of combinatorial Hopf algebras
\cite{LR2}. In particular, it contains the free $2$-$As$ algebra on one generator: this is the Hopf subalgebra $\h_{\WNP}$ of \emph{WN posets},
see definition \ref{3}. Another interesting Hopf subalgebra $\h_{\PP}$ is given by \emph{plane posets}, that is to say double poset with a particular
condition of (in)compatibility between the two orders (definition \ref{2}). \\

We investigate in the present text the algebraic properties of the family of \emph{special posets}, that is to say double posets such that the second order is total
 \cite{MR2}. They generate a Hopf subalgebra of $\h_{\DP}$ denoted by $\h_{\SP}$. 
For example, as explained in \cite{Foissy1}, the two partial orders of a plane poset allow to define a third, total order, so plane posets can also be considered 
as special posets: this defines an injective morphism of Hopf algebras from $\h_{\PP}$ to $\h_{\SP}$. Its image is denoted by $\h_{\SPP}$.
Another interesting Hopf subalgebra of $\h_{\SP}$  is generated by the set of \emph{ordered forests}; 
it is the Hopf algebra $\h_{\OF}$ used in \cite{Foissy3,Foissy7}.  A special poset is \emph{heap-ordered} if its second order (recall it is total) 
is a linear extension of the first one; these objects define another Hopf subalgebra $\h_{\HOP}$ of $\h_{\SP}$.  
Taking the intersections, we finally obtain a commutative diagram of six Hopf algebras:
\[\xymatrix{\h_{\SPP}\ar@{^(->}[r]&\h_{\HOP}\ar@{^(->}[r]&\h_{\SP}\\
\h_{\SPF}\ar@{^(->}[r]\ar@{^(->}[u]&\h_{\HOF}\ar@{^(->}[r]\ar@{^(->}[u]&\h_{\OF}\ar@{^(->}[u]}\]
The Hopf algebra $\h_{\HOF}$ of \emph{heap-ordered forests} is used in \cite{Foissy7}; 
$\h_{\SPF}$ is generated by the set of \emph{plane forests}, considered as special posets, and is isomorphic to the coopposite of the non commutative 
Connes-Kreimer Hopf algebra of plane forests $\h_{\SPF}$ \cite{Foissy4,Foissy5,Holtkamp}. 

A Hopf algebra morphism $\Theta$, from $\h_{\SP}$ to the Malvenuto-Reutenauer Hopf algebra of permutations $\FQSym$ \cite{MR1}, 
also known as the Hopf algebra of free quasi-symmetric functions \cite{Duchamp}, is defined in \cite{MR2}.
This construction uses the linear extensions of the first order of a special poset. The morphism $\Theta$ is surjective and respects the Hopf pairings 
defined on $\h_{\SP}$ and $\FQSym$. Moreover, its restrictions to $\h_{\SPP}$ and $\h_{\HOF}$ are isometric Hopf algebra isomorphisms (corollary \ref{22}).
In the particular case of $\h_{\SPP}$, this is proved using, first a bijection from the set of special plane posets of order $n$ to the $n$-th symmetric group $\S_n$ 
for all $n \geq 0$, then intervals in $\S_n$ for the right weak Bruhat order, see proposition \ref{21}.  As a consequence, we obtain a commutative diagram:
\[\xymatrix{\h_{\SPP}\ar@{_{(}->}[rd] \ar@{^{(}->>}[rrd]&&\\
&\h_{\SP}\ar@{->>}^{\hspace{-5mm}\Theta}[r]&\FQSym\\
\h_{\HOF}\ar@{^{(}->}[ru] \ar@{_{(}->>}[rru]&&}\]
We then complete this diagram with a Hopf algebra morphism $\Upsilon:\h_{\SP} \longrightarrow \h_{\HOF}$,
combinatorially defined (theorem \ref{25}), such that its restriction to $\h_{\SPP}$ gives the following commutative diagram:
\[\xymatrix{\h_{\SPP}\ar@{_{(}->}[rd] \ar@{^{(}->>}[rrd] \ar@{_{(}->>}[dd]_{\Upsilon}&&\\
&\h_{\SP}\ar@{->>}^{\hspace{-5mm}\Theta}[r]&\FQSym\\
\h_{\HOF}\ar@{^{(}->}[ru] \ar@{_{(}->>}[rru]&&}\]
The definition of $\Upsilon$ uses two transformations of special posets, summarized by
$\tddeux{$j$}{$i$}\longrightarrow \tdun{$i$}\tdun{$j$}-\tddeux{$i$}{$j$}$ and
$\pdtroisun{$k$}{$i$}{$j$}\longrightarrow \tddeux{$i$}{$k$} \tdun{$j$}-\tdtroisun{$i$}{$k$}{$j$}+\tdtroisdeux{$i$}{$j$}{$k$}$.

In order to prove the cofreeness of $\h_{\SPF}$, $\h_{\SP}$, $\h_{\HOP}$, $\h_{\SPP}$, $\h_{\OF}$ and $\h_{\SWNP}$,
we introduce a new product $\nwarrow$ on $\h_{\SP}$ making it a \emph{duplicial algebra} \cite{Loday2}, and two non associative coproducts $\Delta_\prec$
and $\Delta_\succ$, making it a \emph{dendriform coalgebra} \cite{Loday1,Ronco}, see paragraph \ref{sect51}. 
These two complementary structures are compatible, and $\h_{\SP}$ is a \emph{Dup-Dend bialgebra} \cite{Foissy3}. 
By the theorem of rigidity for Dup-Dend bialgebras, all these objects are isomorphic to non-commutative Connes-Kreimer Hopf algebras
of decorated plane forests \cite{Foissy4,Foissy5,Holtkamp} (note that this result was obvious for $\h_{\SPF}$), so are free and cofree. 
Moreover, it is possible to define a Dup-Dend structure on $\FQSym$ 
in such a way that the Hopf algebra morphism $\Theta$ becomes a morphism of Dup-Dend bialgebras.
Dendriform structures are also used to show that the restriction of the pairing of $\h_{\DP}$ on $\h_{\SPF}$ is nondegenerate,
with the help of \emph{bidendriform bialgebras} \cite{Foissy6}: in fact, the pairing of $\h_{\SP}$ restricted to $\h_{\SPF}$
respects a certain bidendriform structure. 

In the seventh section, we construct an isometric Hopf algebra morphism between $\h_{\PP}$ and $\h_{\SPP}$.
These two Hopf algebras are clearly isomorphic, with a very easily defined isomorphism, which is not an isometry. 
We prove that these two objects are isometric as Hopf algebras up to two conditions on the base field: it should be not of characteristic two
and should contain a square root of $-1$. This is done using the freeness and cofreeness of $\h_{\PP}$ and 
manipulations of symmetric matrices. \\

This text is organized as follows. The first section recalls the concepts and notations on the Hopf algebra of double posets $\h_{\DP}$.
The second section introduces special posets, heap-ordered posets, special plane posets and the other families of double posets here studied. 
The bijection between the set of special plane posets of order $n$ and $\S_n$ is defined in the third section.
The properties of the morphism $\Theta$ from $\h_{\SP}$ to $\FQSym$ are investigated in the next section.
In particular, it is proved that its restrictions to $\h_{\SPP}$ or $\h_{\HOF}$ are isomorphisms, and the induced isomorphism
from $\h_{\SPP}$ to $\h_{\HOF}$ is combinatorially defined. The fifth and sixth sections introduce duplicial, dendriform and bidendriform structures
and gives applications of these algebraic objects on our families of posets. The problem of finding an isometry from $\h_{\SPP}$ to $\h_{\PP}$ 
is studied in the seventh section; all the obtained results are summarized up in the conclusion. \\

\textbf{Acknowledgements}. The author warmly thanks Darij Grinberg for pointing an error in the preceding version of the paper,
on a lemma on symmetric integral matrices. The proofs of the last section have been changed accordingly.

\begin{notation}
\begin{enumerate}
\item $\K$ is a commutative field. Any algebra, coalgebra, Hopf algebra\ldots of the present text will be taken over $\K$.
\item If $\h=(\h,m,1,\Delta,\varepsilon,S)$ is a Hopf algebra, we shall denote by $\h^+$ its augmentation ideal, that is to say $\Ker(\varepsilon)$.
This ideal $\h^+$ has a coassociative, non counitary coproduct $\tdelta$, defined by $\tdelta(x)=\Delta(x)-x\otimes 1-1\otimes x$ for all $x\in \h^+$.
\item For all $n\geq 1$, $\S_n$ is the $n$-th symmetric group. Any element $\sigma$ of $\S_n$ will be represented by the word $\sigma(1)\ldots \sigma(n)$.
By convention, $\S_0$ is a group with a single element, denoted by the empty word $1$.
\end{enumerate}\end{notation}

\section{Reminders on double posets}

\subsection{Several families of double posets}

\begin{defi}\label{1} \textnormal{
\cite{MR2}. A \textit{double poset} is a triple $(P,\leq_1,\leq_2)$, where $P$ is a finite set and $\leq_1$, $\leq_2$ are two partial orders on $P$.
The set of isoclasses of double posets will be denoted by $\DP$. 
The set of isoclasses of double posets of cardinality $n$ will be denoted by $\DP(n)$ for all $n\in \mathbb{N}$.
}\end{defi}

\begin{remark}
Let $P\in \DP$. Then any subset $Q \subseteq P$ inherits also two partial orders by restriction, so is also a double poset:
we shall speak in this way of \emph{double subposets}.
\end{remark}

\begin{defi}\label{2}\textnormal{
A \textit{plane poset} is a double poset $(P,\leq_h,\leq_r)$ such that for all $x, y\in P$ with $x\neq y$, $x$ and $y$ are comparable for $\leq_h$ if, 
and only if, $x$ and $y$ are not comparable for $\leq_r$. The set of isoclasses of plane posets will be denoted by $\PP$. 
For all $n \in \mathbb{N}$, the set of isoclasses of plane posets of cardinality $n$ will be denoted by $\PP(n)$.
}\end{defi}

If $(P,\leq_h,\leq_r)$ is a plane poset, we shall represent the Hasse graph of $(P,\leq_h)$ 
such that $x <_r y$ in $P$, if and only if $y$ is more on the right than $x$ in the graph. Because of the incompatibility condition between the two orders,
this is a faithful representation of plane posets. For example, let us consider the two following Hasse graphs:
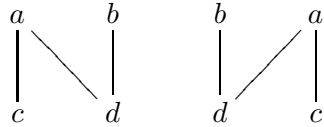
\[\xymatrix{a\ar@{-}[d]\ar@{-}[dr]&b\ar@{-}[d]\\c&d}\hspace{1cm}\xymatrix{b\ar@{-}[d]&a\ar@{-}[d]\ar@{-}[dl]\\d&c}\]
The first one represents the plane poset $(P,\leq_h,\leq_r)$ such that:
\begin{itemize}
\item $\{(x,y) \in P^2\mid x<_h y\}=\{(c,a),(d,a),(d,b)\}$,
\item $\{(x,y) \in P^2\mid x<_r y\}=\{(a,b),(c,b),(c,d)\}$,
\end{itemize}
whereas the second one represents the plane poset $(Q,\leq_h,\leq_r)$ such that:
\begin{itemize}
\item $\{(x,y) \in Q^2\mid x<_h y\}=\{(c,a),(d,a),(d,b)\}$,
\item $\{(x,y) \in Q^2\mid x<_r y\}=\{(b,a),(b,c),(d,c)\}$.
\end{itemize}

\begin{example} 
The empty double poset is denoted by $1$. 
\begin{align*}
\PP(0)&=\{1\},\\
\PP(1)&=\{\tun\},\\
\PP(2)&=\{\tun\tun,\tdeux\},\\
\PP(3)&=\{\tun\tun\tun,\tun\tdeux,\tdeux\tun,\ttroisun,\ttroisdeux,\ptroisun\},\\
\PP(4)&=\left\{\begin{array}{c}
\tun\tun\tun\tun,\tun\tun\tdeux,\tun\tdeux\tun,\tdeux\tun\tun,\tun\ttroisun,\ttroisun\tun,\tun\ttroisdeux,\ttroisdeux\tun,\tun\ptroisun,\ptroisun\tun,\tdeux\tdeux,\tquatreun,\\
\tquatredeux,\tquatretrois,\tquatrequatre,\tquatrecinq,\pquatreun,\pquatredeux,
\pquatretrois,\pquatrequatre,\pquatrecinq,\pquatresix,\pquatresept,\pquatrehuit\end{array}\right\}.
\end{align*}\end{example}

\begin{remark}Let $F$ be a plane forest. We defined in \cite{Foissy4} two partial orders on $F$, which makes it a plane poset:
\begin{itemize}
\item We orient the edges of the forest $F$ from the roots to the leaves. The obtained oriented graph is the Hasse graph of the partial order $\leq_h$.
In other words, if $x,y \in F$, $x\leq_h y$ if, and only if, there is an oriented path from $x$ to $y$ in $F$.
\item if $x,y$ are two vertices of $F$ which are not comparable for $\leq_h$, two cases can occur.
\begin{itemize}
\item If $x$ and $y$ are in two different trees of $F$, then one of these trees is more on the left than the other; this defines the order $\leq_r$ on $x$ and $y$.
\item If $x$ and $y$ are in the same tree $T$ of $F$, as they are not comparable for $\leq_h$ they are both different from the root of $T$.
We then compare them in the plane forest obtained by deleting the root of $T$. 
\end{itemize}
This inductively defines the order $\leq_r$ for any plane forest by induction on the number of vertices.
\end{itemize}
 Equivalently, a plane poset is a plane forest if, and only if its Hasse graph is a forest. The set of plane forests will be denoted by $\PF$; 
for all $n\geq 0$, the set of plane forests with $n$ vertices will be denoted by $\PF(n)$. For example:
\begin{align*}
\PF(0)&=\{1\},\\
\PF(1)&=\{\tun\},\\
\PF(2)&=\{\tun\tun, \tdeux\},\\
\PF(3)&=\{\tun\tun\tun,\tun\tdeux,\tdeux\tun,\ttroisun,\ttroisdeux\},\\
\PF(4)&=\left\{\tun\tun\tun\tun,\tun\tun\tdeux,\tun\tdeux\tun,\tdeux\tun\tun,\tun\ttroisun,\ttroisun\tun,
\tun\ttroisdeux,\ttroisdeux\tun,\tdeux\tdeux,\tquatreun,\tquatredeux,\tquatretrois,\tquatrequatre,\tquatrecinq\right\}.
\end{align*}\end{remark}

\begin{defi}\label{3}\textnormal{
Let $P$ be a double poset. We shall say that $P$ is \textit{WN} ("without N") if it is plane and does not contain any double subposet isomorphic
to $\pquatrecinq$ nor $\pquatresix$. The set of isoclasses of WN posets will be denoted by $\WNP$. 
For all $n \in \mathbb{N}$, the set of isoclasses of WN posets of cardinality $n$ will be denoted by $\WNP(n)$.
}\end{defi}

\begin{example}
\begin{align*}
\WNP(0)&=\{1\},\\
\WNP(1)&=\{\tun\},\\
\WNP(2)&=\{\tun\tun,\tdeux\},\\
\WNP(3)&=\{\tun\tun\tun,\tun\tdeux,\tdeux\tun,\ttroisun,\ttroisdeux,\ptroisun\},\\
\WNP(4)&=\left\{\begin{array}{c}
\tun\tun\tun\tun,\tun\tun\tdeux,\tun\tdeux\tun,\tdeux\tun\tun,\tun\ttroisun,\ttroisun\tun,\tun\ttroisdeux,\ttroisdeux\tun,\tun\ptroisun,\ptroisun\tun,\tdeux\tdeux,\tquatreun,\\
\tquatredeux,\tquatretrois,\tquatrequatre,\tquatrecinq,\pquatreun,\pquatredeux,\pquatretrois,\pquatrequatre,\pquatresept,\pquatrehuit\end{array}\right\}.
\end{align*}\end{example}

\begin{remark}
$\PF \subsetneq \WNP \subsetneq \PP$.
\end{remark}

\subsection{Products and coproducts of double posets}

\begin{defi}\label{4} \textnormal{
Let $P$ and $Q$ be two elements of $\DP$. We define $PQ\in \DP$ by:
\begin{itemize}
\item $PQ$ is the disjoint union of $P$ and $Q$ as a set.
\item $P$ and $Q$ are double subposets of $PQ$.
\item For all $x\in P$, $y\in Q$, $x \leq_2 y$ in $PQ$ and $x$ and $y$ are not comparable for $\leq_1$ in $PQ$.
\end{itemize}}\end{defi}

\begin{remark}\begin{enumerate}
\item This product is called \textit{composition} in \cite{MR2} and denoted by $\prodg$ in \cite{Foissy1}. 
\item The Hasse graph of $PQ$ (in the sense defined below)  is the concatenation of the Hasse graphs of $P$ and $Q$,
that is to say the disjoint union of these graphs, the graph of $P$ being on the left of the graph of $Q$. 
\end{enumerate}\end{remark}

This associative product is linearly extended to the vector space $\h_{\DP}$ generated by the set of double posets. Moreover, the subspaces $\h_{\PP}$,
$\h_{\WNP}$ and $\h_{\PF}$ respectively generated by the sets $\PP$, $\WNP$ and $\PF$ are  stable under this product.

\begin{defi}\label{5}\textnormal{\cite{MR2}. \begin{enumerate}
\item  Let $P=(P,\leq_1,\leq_2)$ be a double poset and let $I\subseteq P$. We shall say that $I$ is a \textit{$1$-ideal} of $P$ if:
\[\forall x\in I,\: \forall y\in P,\: (x \leq_1 y)\Longrightarrow (y\in I).\]
We shall write shortly "ideal" instead of "$1$-ideal" in the sequel.
\item The associative algebra $\h_{\DP}$ is given a Hopf algebra structure with the following coproduct: for any double poset $P$,
\[\Delta(P)=\sum_{\mbox{\scriptsize $I$ ideal of $P$}} (P\setminus I)\otimes I.\]
This Hopf algebra is graded by the cardinality of the double posets. 
\end{enumerate}}\end{defi}

As any double subposet of a, respectively, plane poset, WN poset, plane forest, is also a, respectively, plane poset, WN poset, plane forest,
$\h_{\PP}$, $\h_{\WNP}$ and $\h_{\PF}$ are Hopf subalgebras of $\h_{\DP}$. \\

\begin{example}
\begin{align*}
\tdelta(\tdeux)&=\tun\otimes \tun\\
\tdelta(\ttroisun)&=2\tdeux \otimes \tun+\tun\otimes \tun\tun\\
\tdelta(\ttroisdeux)&=\tun\otimes \tdeux+\tdeux \otimes \tun\\
\tdelta(\ptroisun)&=\tun\tun \otimes \tun+2\tun\otimes \tdeux\\
\tdelta(\tquatreun)&=\tun\otimes \tun\tun\tun+3\tdeux\otimes \tun\tun+3\ttroisun\otimes \tun\\
\tdelta(\tquatredeux)&=\ttroisdeux\otimes \tun+\ttroisun\otimes \tun+\tdeux\otimes \tdeux+\tdeux \otimes \tun\tun+\tun \otimes \tdeux\tun\\
\tdelta(\tquatretrois)&=\ttroisdeux\otimes \tun+\ttroisun\otimes \tun+\tdeux\otimes \tdeux+\tdeux \otimes \tun\tun+\tun \otimes \tun\tdeux\\
\tdelta(\tquatrequatre)&=2\ttroisdeux \otimes \tun+\tun\otimes \ttroisun+\tdeux \otimes \tun\tun\\
\tdelta(\tquatrecinq)&=\tun\otimes \ttroisdeux+\tdeux \otimes \tdeux+\ttroisdeux \otimes \tun\\
\tdelta(\pquatreun)&=\tun\tun\tun\otimes\tun+3\tun\tun\otimes\tdeux+3\tun \otimes \ptroisun\\
\tdelta(\pquatredeux)&=\tun\otimes\ttroisdeux + \tun\otimes\ptroisun+\tdeux\otimes \tdeux+\tun\tun \otimes \tdeux+\tdeux \tun \otimes \tun\\
\tdelta(\pquatretrois)&=\tun\otimes\ttroisdeux +\tun\otimes\ptroisun+\tdeux\otimes \tdeux+\tun\tun\otimes \tdeux +\tun \tdeux \otimes \tun\\
\tdelta(\pquatrequatre)&=2\tun \otimes \ttroisdeux +\ptroisun\otimes \tun+ \tun\tun \otimes \tdeux\\
\tdelta(\pquatrecinq)&=\tdeux\tun \otimes \tun+\ptroisun \otimes \tun+\tdeux \otimes \tdeux+\tun\tun\otimes \tun\tun+\tun\otimes \tun\tdeux+\tun \otimes \ttroisun\\
\tdelta(\pquatresix)&=\tun\tdeux \otimes \tun+\ptroisun \otimes \tun+\tdeux \otimes \tdeux+\tun\tun\otimes \tun\tun+\tun\otimes \tdeux \tun+\tun \otimes \ttroisun\\
\tdelta(\pquatresept)&=2\ptroisun\otimes \tun+2\tun \otimes \ttroisun+\tun\tun\otimes \tun\tun\\
\tdelta(\pquatrehuit)&=\ttroisun \otimes \tun+2\tdeux \otimes \tdeux+\tun\otimes \ptroisun
\end{align*}
\end{example}

\subsection{Hopf pairing on double posets}

\begin{defi}\label{6}\textnormal{\cite{MR2}\begin{enumerate}
\item For two double posets $P,Q$,  $S(P,Q)$ is the set of bijections $\sigma:P\longrightarrow Q$ such that, for all $i,j \in P$:
\begin{itemize}
\item ($i\leq_1 j$ in $P$) $\Longrightarrow$ ($\sigma(i) \leq_2 \sigma(j)$ in $Q$).
\item ($\sigma(i)\leq_1 \sigma(j)$ in $Q$) $\Longrightarrow$ ($i \leq_2 j$ in $P$).
\end{itemize}
These bijections are called \emph{pictures}.
\item We define a pairing on $\h_{\DP}$ by $\langle P,Q \rangle=Card(S(P,Q))$ for $P,Q \in \DP$. This pairing is a symmetric Hopf pairing.
\end{enumerate}}\end{defi}

It is proved in \cite{Foissy1} that this pairing is nondegenerate if, and only if, the characteristic of $\K$ is zero.
Moreover, the restriction of this pairing to $\h_{\PP}$, $\h_{\PF}$ or $\h_{\WNP}$ is nondegenerate, whatever the field $\K$ is.

\section{Several families of posets}

\subsection{Special posets}

\begin{defi}\textnormal{
\cite{MR2}. A double poset $P=(P,\leq_1,\leq_2)$ is \textit{special} if the order $\leq_2$ is total. The set of special double posets will be denoted by $\SP$.
The set of special double posets of cardinality $n$ will be denoted by $\SP(n)$.
}\end{defi}

This notion is equivalent to the notion of labeled posets. If $(P,\leq_1,\leq_2)$ is a special poset of order $n$, there is a unique isomorphism from
$(P,\leq_2)$ to $(\{1,\ldots,n\}, \leq)$, and we shall often identify them.\\

\begin{example}
We shall graphically represent a special poset $(P,\leq_1,\leq_2)$ by the Hasse graph of $(P,\leq_1)$, with indices on the vertices
giving the total order $\leq_2$.
\begin{enumerate}
\item Here are $\SP(n)$ for $n\leq 3$:
\begin{align*}
\SP(0)&=\{1\},\\
\SP(1)&=\{\tdun{$1$}\},\\
\SP(2)&=\{\tdun{$1$}\tdun{$2$},\tddeux{$1$}{$2$},\tddeux{$2$}{$1$}\},\\
\SP(3)&=\left\{\begin{array}{c}
\tdun{$1$}\tdun{$2$}\tdun{$3$},\tdun{$1$}\tddeux{$2$}{$3$},\tdun{$1$}\tddeux{$3$}{$2$},\tdun{$2$}\tddeux{$1$}{$3$},
\tdun{$2$}\tddeux{$3$}{$1$},\tdun{$3$}\tddeux{$1$}{$2$},\tdun{$3$}\tddeux{$2$}{$1$},\\
\tdtroisun{$1$}{$3$}{$2$},\tdtroisun{$2$}{$3$}{$1$},\tdtroisun{$3$}{$2$}{$1$},
\pdtroisun{$1$}{$2$}{$3$},\pdtroisun{$2$}{$1$}{$3$},\pdtroisun{$3$}{$1$}{$2$},
\tdtroisdeux{$1$}{$2$}{$3$},\tdtroisdeux{$1$}{$3$}{$2$},\tdtroisdeux{$2$}{$1$}{$3$},\tdtroisdeux{$2$}{$3$}{$1$},
\tdtroisdeux{$3$}{$1$}{$2$},\tdtroisdeux{$3$}{$2$}{$1$}
\end{array}\right\}.\end{align*}
\item See \cite{Foissy7}. \emph{Ordered forests} are special double posets. The set of ordered forests will be denoted by $\OF$.
The set of ordered forests of cardinality $n$ will be denoted by $\OF(n)$. For example:
\begin{align*}
\OF(0)&=\{1\},\\
\OF(1)&=\{\tdun{$1$}\},\\
\OF(2)&=\{\tdun{$1$}\tdun{$2$},\tddeux{$1$}{$2$},\tddeux{$2$}{$1$}\},\\
\OF(3)&=\left\{\begin{array}{c}
\tdun{$1$}\tdun{$2$}\tdun{$3$},\tdun{$1$}\tddeux{$2$}{$3$},\tdun{$1$}\tddeux{$3$}{$2$},\tdun{$2$}\tddeux{$1$}{$3$},
\tdun{$2$}\tddeux{$3$}{$1$},\tdun{$3$}\tddeux{$1$}{$2$},\tdun{$3$}\tddeux{$2$}{$1$},\\
\tdtroisun{$1$}{$3$}{$2$},\tdtroisun{$2$}{$3$}{$1$},\tdtroisun{$3$}{$2$}{$1$},
\tdtroisdeux{$1$}{$2$}{$3$},\tdtroisdeux{$1$}{$3$}{$2$},\tdtroisdeux{$2$}{$1$}{$3$},\tdtroisdeux{$2$}{$3$}{$1$},
\tdtroisdeux{$3$}{$1$}{$2$},\tdtroisdeux{$3$}{$2$}{$1$}
\end{array}\right\}.\end{align*}

\item Let $P=(P,\leq_h,\leq_r)$ be a plane poset. From proposition 11 in \cite{Foissy1}, the relation $\leq$ defined by $x \leq y$ if, and only if, $x\leq_h y$ 
or $x\leq_r y$, is a total order on $P$, called the \textit{induced total order} on $P$. So $(P,\leq_h,\leq)$ is also a special double poset: we can consider
plane posets as special posets. The set of plane posets, seen as special double posets, will be denoted by $\SPP$.
The set of plane posets of cardinality $n$, seen as special double posets, will be denoted by $\SPP(n)$. For example:
\begin{align*}
\SPP(0)&=\{1\},\\
\SPP(1)&=\{\tdun{$1$}\},\\
\SPP(2)&=\{\tdun{$1$}\tdun{$2$},\tddeux{$1$}{$2$}\},\\
\SPP(3)&=\left\{\begin{array}{c}
\tdun{$1$}\tdun{$2$}\tdun{$3$},\tdun{$1$}\tddeux{$2$}{$3$},\tddeux{$1$}{$2$}\tdun{$3$},
\tdtroisun{$1$}{$3$}{$2$},\pdtroisun{$3$}{$1$}{$2$},\tdtroisdeux{$1$}{$2$}{$3$}
\end{array}\right\}.\end{align*}

\item We define the set $\SPF$ of plane forests, seen as special posets, and the set $\SWNP$ of WN posets, seen as special posets.
Note that $\SPF=\OF \cap \SPP$. For example:
\begin{align*}
\SPF(0)&=\{1\},\\
\SPF(1)&=\{\tdun{$1$}\},\\
\SPF(2)&=\{\tdun{$1$}\tdun{$2$},\tddeux{$1$}{$2$}\},\\
\SPF(3)&=\left\{\begin{array}{c}
\tdun{$1$}\tdun{$2$}\tdun{$3$},\tdun{$1$}\tddeux{$2$}{$3$},\tddeux{$1$}{$2$}\tdun{$3$},
\tdtroisun{$1$}{$3$}{$2$},\tdtroisdeux{$1$}{$2$}{$3$}
\end{array}\right\}\end{align*}\end{enumerate}\end{example}

If $P$ and $Q$ are special double posets, then $PQ$ is also special. So the space $\h_{\SP}$ generated by special double posets
is a subalgebra of $(\h_{\DP},\prodg)$. Moreover, if $P$ is a special double poset, then any subposet of $P$ is also special.
As a consequence, $\h_{\SP}$ is a Hopf subalgebra of $\h_{\DP}$; this Hopf algebra also appears in \cite{Schocker}. 
Similarly, the spaces $\h_{\OF}$, $\h_{\SPP}$, $\h_{\SWNP}$ and $\h_{\SPF}$  
generated by $\OF$, $\SPP$, $\SWNP$ and $\SPF$ are Hopf subalgebras of $\h_{\DP}$. \\

\begin{remark}
It is clear that $\h_{\PP}$ and $\h_{\SPP}$ are isomorphic Hopf algebras, via the isomorphism sending the plane poset 
$(P,\leq_h,\leq_r)$ 
to the special poset $(P,\leq_h,\leq)$. The same argument works for $\h_{\WNP}$ and $\h_{\SWNP}$, and for $\h_{\PF}$ and $\h_{\SPF}$.
\end{remark}

\subsection{Heap-ordered posets}

\begin{defi}\textnormal{
Let $P=(P,\leq_1,\leq_2)$ be a special double poset. It is \textit{heap-ordered} if for all $x,y\in P$, $x\leq_1 y$ implies that $x\leq_2 y$.
The set of heap-ordered posets will be denoted by $\HOP$. The set of heap-ordered posets of cardinality $n$ will be denoted by $\HOP(n)$.
We put $\HOF=\HOP\cap \OF$ and $\HOF(n)=\HOP(n) \cap \OF(n)$ for all $n$.
}\end{defi}

\begin{example}
Here are the sets $\HOP(n)$ and $\HOF(n)$ for $n\leq 3$:
\begin{align*}
\HOP(1)&=\{\tdun{$1$}\},\\
\HOP(2)&=\{\tdun{$1$}\tdun{$2$},\tddeux{$1$}{$2$}\},\\
\HOP(3)&=\left\{\begin{array}{c}
\tdun{$1$}\tdun{$2$}\tdun{$3$},\tdun{$1$}\tddeux{$2$}{$3$},\tdun{$2$}\tddeux{$1$}{$3$},\tdun{$3$}\tddeux{$1$}{$2$},
\tdtroisun{$1$}{$3$}{$2$},\pdtroisun{$3$}{$1$}{$2$},\tdtroisdeux{$1$}{$2$}{$3$}
\end{array}\right\},\\ \\
\HOF(1)&=\{\tdun{$1$}\},\\
\HOF(2)&=\{\tdun{$1$}\tdun{$2$},\tddeux{$1$}{$2$}\},\\
\HOF(3)&=\left\{\begin{array}{c}
\tdun{$1$}\tdun{$2$}\tdun{$3$},\tdun{$1$}\tddeux{$2$}{$3$},\tdun{$2$}\tddeux{$1$}{$3$},\tdun{$3$}\tddeux{$1$}{$2$},
\tdtroisun{$1$}{$3$}{$2$},\tdtroisdeux{$1$}{$2$}{$3$}
\end{array}\right\}.\end{align*}\end{example}

Note that $\SPP \subsetneq \HOP$ and $\SPF \subsetneq \HOF$, as $\tdun{$2$}\tddeux{$1$}{$3$}$ is not a plane poset.
It is well-known that $|\HOF(n)|=n!$ for all $n\geq 0$.

If $P$ and $Q$ are two heap-ordered posets, then $PQ$ also is. As a consequence, the spaces $\h_{\HOP}$, $\h_{\HOF}$ and $\h_{\SPF}$
generated by $\HOP$, $\HOF$ and $\SPF$ are Hopf subalgebras of $\h_{\DP}$.
Moreover, plane posets are heap-ordered, so $\h_{\SPP}\subseteq \h_{\HOP}$. We obtain a commutative diagram of canonical injections:
\[\xymatrix{\h_{\SPP}\ar@{^(->}[r]&\h_{\HOP}\ar@{^(->}[r]&\h_{\SP}\\
\h_{\SPF}\ar@{^(->}[r]\ar@{^(->}[u]&\h_{\HOF}\ar@{^(->}[r]\ar@{^(->}[u]&\h_{\OF}\ar@{^(->}[u]}\]

\begin{prop}\label{9}\begin{enumerate}
\item Let $P \in \SP$. Then $P$ is heap-ordered if, and only if, it does not contain any double subposet isomorphic to $\tddeux{$2$}{$1$}$.
\item Let $P \in \SP$.  Then $P\in \SPP$ if, and only if, it does not contain any double subposet isomorphic to $\tddeux{$1$}{$3$} \tdun{$2$}$ 
nor $\tddeux{$2$}{$1$}$.
\end{enumerate} \end{prop}

\begin{proof} The first point is immediate. \\

$2.\Longrightarrow$. If $P \in \SPP$, then any subposet of $P$ belongs to $\SPP$. The conclusion comes
from the fact that $\tddeux{$1$}{$3$} \tdun{$2$}$ and $\tddeux{$2$}{$1$}$ are not special plane posets.\\

$2. \Longleftarrow$. By the first point, $P=(P,\leq_1,\leq_2)$ is heap-ordered. We define a relation $\leq_r$ on $P$ by:
\[x\leq_r y \mbox{ if }(x=y)\mbox{ or }\left((x<_2 y) \mbox{ and not } (x <_1 y)\right).\]
By definition, $x\leq_2 y$ if, and only if, $x\leq_1 y$ or $x\leq_r y$. Moreover, if $x$ and $y$ are comparable for both $\leq_1$ and $\leq_r$,
then $x=y$ by definition of $\leq_r$. It remains to prove that $\leq_r$ is a partial order on $P$.
If $x <_r y$ and $y<_r z$, then $x<_2 y<_2 z$, so $x<_2 z$, so $x<_1 z$ or $x<_r z$. If $x <_1 z$, then the subposet $\{x,y,z\}$ of $P$ 
is equal to $\tddeux{$1$}{$3$} \tdun{$2$}$, as $x,y$ and $y,z$ are not comparable for $\leq_1$: contradiction. So $x<_r z$. \end{proof}

\subsection{Pairing on special posets}

We restrict the pairing of $\h_{\DP}$ to $\h_{\SP}$. The matrix of the restriction of this pairing to $\h_{\SP}(2)$ is:
\[\begin{array}{c|c|c|c}
&\tdun{$1$}\tdun{$2$}&\tddeux{$1$}{$2$}&\tddeux{$2$}{$1$}\\
\hline \tdun{$1$}\tdun{$2$}&2&1&1\\
\hline \tddeux{$1$}{$2$}&1&1&0\\
\hline \tddeux{$2$}{$1$}&1&0&1
\end{array}\]

\begin{remark}
\begin{enumerate}
\item As a consequence, $\tdun{$1$}\tdun{$2$}-\tddeux{$1$}{$2$}-\tddeux{$2$}{$1$}$ is in the kernel of the pairing.
Hence, $\langle-,-\rangle_{\mid \h_{\SP}}$, $\langle-,-\rangle_{\mid \h_{\HOP}}$ and $\langle-,-\rangle_{\mid \h_{\OF}}$ are degenerate.
The kernels of these restrictions of the pairing are described in corollary \ref{19}.
\item A direct (but quite long) computation shows that the following element is in the kernel of  $\langle-,-\rangle_{\mid \h_{\SWNP}}$:
\begin{align*}
&\pquatrehuit-\tquatredeux-\tquatretrois+\tquatreun+\pquatresept-\pquatredeux-\pquatretrois\\
&+\ttroisdeux\tun-\ttroisun\tun+\pquatreun-\ptroisun\tun+\tdeux\tdeux+\tun\ttroisdeux-\tun\ttroisun-\tun\ptroisun+\tun\tdeux\tun.
\end{align*}
(We write here the double posets appearing in this element as plane poset, they have to be considered as special posets).
So $\langle-,-\rangle_{\mid \h_{\SWNP}}$ is degenerate.
\item We shall see that $\langle-,-\rangle_{\mid \h_{\HOF}}$, $\langle-,-\rangle_{\mid \h_{\SPP}}$ 
and $\langle-,-\rangle_{\mid \h_{\SPF}}$ are nondegenerate, see corollaries \ref{23}, \ref{26} and \ref{37}.
\end{enumerate}\end{remark}

\section{Links with permutations}

\subsection{Plane poset associated to a permutation}

\begin{prop}\label{10}
Let $\sigma\in \S_n$. We define two relations $\leq_h$ and $\leq_r$ on $\{1,\cdots,n\}$ by:
\begin{itemize}
\item ($i\leq_h j$) if ($i\leq j$ and $\sigma(i)\leq \sigma(j)$).
\item ($i\leq_r j$) if ($i\leq j$ and $\sigma(i)\geq \sigma(j)$).
\end{itemize}
Then $(\{1,\cdots,n\},\leq_h,\leq_r)$ is a plane poset. The induced total order on $\{1,\cdots,n\}$ is the usual total order.
\end{prop}

\begin{proof} It is clear that $\leq_h$ and $\leq_r$ are two partial orders on $\{1,\cdots,n\}$. It is immediate for any $i,j$, $i$ and $j$ are comparable 
for $\leq_h$ or $\leq_r$. Moreover, if $i$ and $j$ are comparable for both $\leq_h$ and $\leq_r$, then $\sigma(i)=\sigma(j)$, so $i=j$.
For all $i,j$, $i\leq_h j$ or $i\leq_r j$ if, and only if, $i\leq j$. \end{proof}

\begin{defi}\textnormal{Let $n\in \mathbb{N}$. We define a map:
\[\Phi_n:\left\{ \begin{array}{rcl}
\S_n&\longrightarrow & \PP(n)\\
\sigma&\longrightarrow &\left(\{1,\cdots,n\},\leq_h,\leq_r\right), \end{array}\right.\]
where $\leq_h$ and $\leq_r$ are defined in proposition \ref{10}.}
\end{defi}

\begin{example}
\begin{align*}
1&\longrightarrow\tun& 12&\longrightarrow\tdeux& 21&\longrightarrow\tun\tun\\
123&\longrightarrow\ttroisdeux& 132&\longrightarrow\ttroisun& 213&\longrightarrow\ptroisun\\
231&\longrightarrow\tdeux\tun& 312&\longrightarrow\tun\tdeux& 321&\longrightarrow\tun\tun\tun\\
1234&\longrightarrow\tquatrecinq& 1243&\longrightarrow\tquatrequatre& 1324&\longrightarrow\pquatrehuit\\
1342&\longrightarrow\tquatredeux& 1423&\longrightarrow\tquatretrois& 1432&\longrightarrow\tquatreun\\
2134&\longrightarrow\pquatrequatre& 2143&\longrightarrow\pquatresept& 2314&\longrightarrow\pquatredeux\\
2341&\longrightarrow\ttroisdeux\tun& 2413&\longrightarrow\pquatrecinq& 2431&\longrightarrow\ttroisun\tun\\
3124&\longrightarrow\pquatretrois& 3142&\longrightarrow\pquatresix& 3214&\longrightarrow\pquatreun\\
3241&\longrightarrow\ptroisun\tun& 3412&\longrightarrow\tdeux\tdeux& 3421&\longrightarrow\tdeux\tun\tun\\
4123&\longrightarrow\tun\ttroisdeux& 4132&\longrightarrow\tun\ttroisun& 4213&\longrightarrow\tun\ptroisun\\
4231&\longrightarrow\tun\tdeux\tun& 4312&\longrightarrow\tun\tun\tdeux& 4321&\longrightarrow\tun\tun\tun\tun
\end{align*}
\end{example}

We shall prove in the next section that $\Phi_n$ is bijective for all $n \geq 0$.

\subsection{Permutation associated to a plane poset}

We now construct the inverse bijection. For any $P\in \PP$, nonempty, we put:
\[\kappa(P)=\max(\{y\in P\:/\: \forall x\in P,\:x\leq y\Rightarrow x \leq_h y\}).\]
Note that $\kappa(P)$ is well-defined: the smallest element of $P$ for its total order belongs to the set $\{y\in P\:/\: \forall x\in P,\:x\leq y\Rightarrow x \leq_h y\}$.\\

Let $P \in \PP(n)$. Up to a unique increasing bijection, we can suppose that $P=\{1,\cdots,n\}$ as a totally ordered set:
we shall take this convention in this paragraph. We define an element $\sigma$ of $\S_n$ by:
\[ \left\{ \begin{array}{rcl}
\sigma^{-1}(n)&=&\kappa(P)\\
\sigma^{-1}(n-1)&=&\kappa\left(P-\{\sigma^{-1}(n)\}\right),\\
\vdots&&\vdots\\
\sigma^{-1}(1)&=&\kappa\left(P-\{\sigma^{-1}(n),\cdots, \sigma^{-1}(2)\}\right).
\end{array}\right.\]
This defines a map:
\[\Psi_n: \left\{ \begin{array}{rcl}
\PP(n)&\longrightarrow & \S_n\\
(P,\leq_h,\leq_r)&\longrightarrow & \sigma.
\end{array}\right.\]

\begin{lemma}
$\Psi_n \circ \Phi_n=Id_{\S_n}$.
\end{lemma}

\begin{proof} Let $\sigma \in \S_n$. We put $P=\Phi_n(\sigma)$ and $\tau=\Psi_n(P)$. Then:
\[\{y\in P\:/\: \forall x\in P,\:x\leq y\Rightarrow x \leq_h y\}=\{j\in\{1,\cdots,n\}\:/\: \forall 1\leq i\leq n,\:i\leq j\Rightarrow \sigma(i) \leq \sigma(j)\}.\]
So $\tau^{-1}(n)=\kappa(P)=\sigma^{-1}(n)$. Iterating this process, we obtain $\sigma^{-1}=\tau^{-1}$, so $\sigma=\tau$. \end{proof}

\begin{lemma} \label{13}
Let $P\in \PP(n)$. We put $\Psi_n(P)=\sigma$. If $i \leq_h j$ in $P$, then $\sigma(i)\leq \sigma(j)$.
\end{lemma}

\begin{proof}If $i=j$, this is obvious. Let us assume that $i<_h j$. We put $k=\sigma(i)$ and $l=\sigma(j)$. Then $k\neq l$. Let us assume that $k>l$. We then put:
\[P'=P\setminus \{\sigma^{-1}(n),\ldots, \sigma^{-1}(k+1)\}=\{i_1,\cdots,i_p,i,i_{p+1},\cdots,i_{p+q},j,i_{p+q+1},\cdots,i_{p+q+r}\},\]
with $i_1<\cdots<i_p<i<i_{p+1}<\cdots<i_{p+q}<j<i_{p+q+1}<\cdots<i_{p+q+r}$. Indeed, as $l<k<k+1$, both $\sigma^{-1}(k)=i$ 
and $\sigma^{-1}(l)=j$ belongs to this set. As $\kappa(P')=i$, $i_1,\cdots,i_p <_h i$. If $i\leq_h i_{p+1}$, then $\kappa(P')\geq i_{p+1}>i$: contradiction. 
So $i<_r i_{p+1}$. \\

Let us prove by induction on $s$ that $i_{p+s}\leq_h j$ for $1\leq s \leq q$. If $i_{p+1}\leq_r j$, then $i$ and $j$ would be comparable for $\leq_r$, 
so would not be comparable for $\leq_h$: contradiction. So $i_{p+1}\leq_h j$. Let us suppose that $i_{p+s-1} \leq_h j$, $1<s\leq q$. 
As $i_{p+s}<j$, $i_{p+s}<_h j$ or $i_{p+s}<_r j$. Let us assume that $i_{p+s}<_r j$. As $\kappa(P')=i<i_{p+s}$, there exists $x \in P'$, $x <_r i_{p+s}$.
By the induction hypothesis, $x \notin \{i_{p+1},\cdots,i_{p+s}\}$. As $i<_h j$, $x\neq i$, so $x\in \{i_1,\cdots,i_p\}$. But for such an $x$, $x<_h i<_h j$, 
so $x<_h j$: contradiction. So $i_{p+s}<_h j$. \\

Finally, we obtain that $i_1,\cdots,i_p,i,i_{p+1},\cdots,i_{p+q},j \leq_h j$, so $i=\kappa(P')\geq j$: contradiction, $i<j$. So $k<l$. \end{proof}

\begin{lemma}
$\Phi_n \circ \Psi_n=Id_{\PP_n}$.
\end{lemma}

\begin{proof} Let $P\in \PP_n$. We put $\sigma=\Psi_n(P)$ and $Q=\Phi_n(\sigma)$. As totally ordered sets, $P=Q=\{1,\cdots,n\}$. 
As they are both plane posets, it is enough to prove that $(P,\leq_h)=(Q,\leq_h)$. Let us suppose that $i \leq_h j$ in $P$. 
Then $i\leq j$ and $\sigma(i)\leq \sigma(j)$ by lemma \ref{13}. So $i\leq_h j$ in $Q$. Let us suppose that $i \leq_h j$ in $Q$. So $i\leq j$ and 
$\sigma(i)\leq \sigma(j)$. We put $k=\sigma(i)$ and $l=\sigma(j)$. As $k<l$: 
\[i \in P'=P-\{\sigma^{-1}(n),\cdots,\sigma^{-1}(l+1)\}.\]
By definition of $\kappa(P')=j$, $i\leq_h j$ in $P$ as $i\leq j$. \end{proof}

\begin{prop} 
$\Psi_n$ is a bijection, of inverse $\Phi_n$. As a consequence, $\Card(\PP(n))=n!$ for all $n\in \mathbb{N}$.
\end{prop}

Here are examples of properties of the bijection $\Psi_n$:

\begin{prop}
Let $P=(P,\leq_h,\leq_r) \in \PP(n)$. 
\begin{enumerate}
\item $n\cdots 1\circ \Psi_n(P)=\Psi_n((P,\leq_r,\leq_h))$.
\item $\Psi_n(P)^{-1}=\Psi_n((P,\leq_h,\geq_r))$.
\end{enumerate}\end{prop}

\begin{proof}  1. We put $\Psi_n(P)=\sigma=(a_1\cdots a_n)$. Then $n\cdots 1\circ \sigma=(n-a_1+1) \cdots (n-a_n+1)$.
We put $Q=\Phi_n(n\cdots 1\circ \sigma)$. For all $i,j\in \{1,\cdots n\}$:
\begin{align*}
i\leq_h j \mbox{ in }Q&\Longleftrightarrow i\leq j \mbox{ and } n-a_i+1\leq n-a_j+1\\
&\Longleftrightarrow i\leq j \mbox{ and } a_i\geq a_j\\
&\Longleftrightarrow  i\leq_r j \mbox{ in } P
\end{align*}
Similarly, $i\leq_r j$ in $Q$ if, and only if, $i\leq_h j$ in $P$. So $Q=(P,\leq_r,\leq_h)$. \\

2. We put $R=\Phi_n(\sigma^{-1})$. Let $i,j \in \{1,\cdots,n\}$.
\begin{align*}
\sigma(i)\leq_h \sigma(j) \mbox{ in }R&\Longleftrightarrow \sigma(i)\leq \sigma(j) \mbox{ and } i\leq j\\
&\Longleftrightarrow i\leq_h j\mbox{ in }P,\\ \\
\sigma(i)\leq_r \sigma(j) \mbox{ in }R&\Longleftrightarrow \sigma(i)\leq \sigma(j) \mbox{ and } i\geq j\\
&\Longleftrightarrow i\geq_r j\mbox{ in }P.
\end{align*}
So $\sigma:(P,\leq_h,\geq_r) \longrightarrow R$ is an isomorphism of plane posets. \end{proof}

\begin{remark}
In other terms, $n\cdots 1\circ \Psi_n(P)=\Psi_n\circ \iota(P)$, where the involution $\iota$ is defined in \cite{Foissy1}
by $\iota((P,\leq_h,\leq_r))=(P,\leq_r,\leq_h)$.
\end{remark}

\section{A morphism to $\FQSym$}

Note that $\h_{\PP}$, $\h_{\SPP}$ and $\FQSym$ are both free and cofree, with the same formal series.
From a result of \cite{Foissy2},  $\h_{\PP}$, hence $\h_{\SPP}$, is isomorphic to $\FQSym$. Our aim in this section is to define and study
an explicit isomorphism between $\h_{\SPP}$ and $\FQSym$.

\subsection{Reminders on $\FQSym$}

Let us first recall the construction of $\FQSym$ \cite{MR1,Duchamp}. As a vector space, a basis of $\FQSym$ is given by the disjoint union 
of the symmetric groups $\S_n$, for all $n \geq 0$. By convention, the unique element of $\S_0$ is denoted by $\emptyset$. The product of $\FQSym$ is given, 
for $\sigma \in \S_k$, $\tau \in \S_l$, by:
\[\sigma\tau=\sum_{\epsilon \in Sh(k,l)} (\sigma \otimes \tau) \circ \epsilon,\]
where $Sh(k,l)$ is the set of $(k,l)$-shuffles, that is to say permutations $\epsilon \in \S_{k+l}$ such that $\epsilon^{-1}(1)<\ldots <\epsilon^{-1}(k)$
and $\epsilon^{-1}(k+1)<\ldots<\epsilon^{-1}(k+l)$.
In other words, the product of $\sigma$ and $\tau$ is given by shifting the letters of the word
representing $\tau$ by $k$, and then summing all the possible shufflings of this word and of the word representing $\sigma$. For example:
\begin{align*}
132.21&=13254+13524+15324+51324+13542\\
&+15342+51342+15432+51432+54132.
\end{align*}

Let $\sigma \in \S_n$. For all $0\leq k \leq n$, there exists a unique triple 
$\left(\sigma_1^{(k)},\sigma_2^{(k)},\zeta_k\right)\in \S_k \times \S_{n-k} \times Sh(k,n-k)$
such that $\sigma=\zeta_k^{-1} \circ \left(\sigma_1^{(k)} \otimes \sigma_2^{(k)}\right)$. The coproduct of $\FQSym$ is then defined by:
\[\Delta(\sigma)=\sum_{k=0}^n \sigma_1^{(k)} \otimes \sigma_2^{(k)}.\]
For example:
\[\Delta(41325)=\emptyset \otimes 41325+1 \otimes 1324+21 \otimes 213+312\otimes 12+4132 \otimes 1+41325 \otimes \emptyset.\]
Note that $\sigma_1^{(k)}$ and $\sigma_2^{(k)}$ are obtained by cutting the word representing $\sigma$ between the $k$-th and the $k+1$-th letter,
and then \textit{standardizing} the two obtained words, that is to say applying to their letters the unique increasing bijection to $\{1,\ldots,k\}$ or $\{1,\ldots,n-k\}$.
Moreover, $\FQSym$ has a nondegenerate, homogeneous, Hopf pairing defined by $\langle \sigma,\tau\rangle=\delta_{\sigma,\tau^{-1}}$
for all permutations $\sigma$ and $\tau$.

\subsection{Linear extensions}

\begin{defi}\textnormal{
Let $P=(P,\leq_1,\leq_2)$ be a special poset. Let $x_1<_2\ldots<_2 x_n$ be the elements of $P$.
A \emph{linear extension} of $P$ is a permutation $\sigma \in \S_n$ such that, for all $i,j \in \{1,\ldots,n\}$:
\[(x_i\leq_1 x_j)\Longrightarrow (\sigma^{-1}(i)<\sigma^{-1}(j)).\]
The set of linear extensions of $P$ will be denoted by $S_P$.
}\end{defi}

\begin{remark}\begin{enumerate}
\item Let $P$ be a special poset. It is heap-ordered if, and only if, $Id_n \in S_P$.
\item Let $P$ be a special poset of cardinality $n$. By definition of the product of plane posets, the plane poset $\tun^n$,
seen as a special poset, has $n$ vertices. If $i\neq j$ in $\tun^n$, then $i$ and $j$ are not comparable for $\leq_h$.
We also identify $P$ and $\tun^n$ with $\{1,\ldots,n\}$ as totally ordered sets. If $\sigma$ is a bijection from $P$ to $\tun^n$, then $\sigma\in S(\tun^n,P)$ if, and only if,
$\sigma(i)<_h \sigma(j) $ in $P$ implies that to $i<j$. Hence, the set of linear extensions of $P$ is $S(\tun^n,P)$. 
\item Let $P$ be a special poset. We denote by $n$ its cardinality. As the second order of $P$ is total, we can identify $P$ with $\{1,\ldots,n\}$,
as totally ordered sets. By \cite{Trotter}, seeing orders on $P$ as elements of $P \times P$:
\[\{(x,y)\in P^2\mid x<_1y\}=\bigcap\{\ll\mid \ll \mbox{ total order extending }\leq_1\}.\]
We identify the total order $i_1\ll\ldots \ll i_n$ on $P$ with the permutation $i_1\ldots i_n$. Then permutations corresponding to total orders extending $\leq_1$
are precisely the elements of $S_P$. We obtain:
\[\{(x,y)\in P^2\mid x<_1y\}=\{(i,j) \in \{1,\ldots,n\}^2\mid \forall \sigma \in S_P, \sigma^{-1}(i)<\sigma^{-1}(j)\}.\]
So $S_P$ entirely determines $P$. 
\end{enumerate}\end{remark}

The following theorem is proved in  \cite{MR2}:

\begin{theo}
 The following map is a surjective morphism of Hopf algebras:
\[\Theta:\left\{\begin{array}{rcl}
\h_{\SP}&\longrightarrow&\FQSym\\
P\in \SP&\longrightarrow&\displaystyle \sum_{\sigma \in S_P} \sigma.
\end{array}\right.\]
Moreover, for any $x,y \in \h_{\SP}$, $\langle x,y \rangle=\langle \Theta(x),\Theta(y)\rangle_\FQSym$.
\end{theo}

\begin{example}
If $\{i,j,k\}=\{1,2,3\}$:
\begin{align*}
\Theta(\tdun{$i$}\tdun{$j$}\tdun{$k$})&=ijk+ikj+jik+jki+kij+kji\\
\Theta(\tdun{$i$}\tddeux{$j$}{$k$})&=ijk+jik+jki\\
\Theta(\tdtroisun{$i$}{$k$}{$j$})&=ijk+ikj\\
\Theta(\tdtroisdeux{$i$}{$j$}{$k$})&=ijk
\end{align*}\end{example}

It is proved in \cite{Foissy7} that the restriction of $\Theta_{\mid \h_{\HOF}}$  is an isomorphism
from $\h_{\HOF}$ to $\FQSym$ (Proposition 7). Consequently, 
$\Theta$ and its restrictions to $\h_{\HOP}$ and to $\h_{\OF}$ are surjective.

\begin{cor}\label{19}
The kernel of the pairing on $\h_{\SP}$ is $\Ker(\Theta)$. 
The kernel of the pairing restricted to $\h_{\HOP}$ and $\h_{\OF}$  is respectively $\Ker(\Theta)\cap \h_{\HOP}$ and $\Ker(\Theta)\cap \h_{\OF}$.
\end{cor}

\begin{proof} For any $x \in \h_{\SP}$, as $\Theta$ is surjective:
\begin{align*}
x \in \h_{\SP}^\perp&\Longleftrightarrow\forall y\in \h_{\SP}, \: \langle x,y\rangle=0\\
&\Longleftrightarrow\forall y\in \h_{\SP}, \: \langle \Theta(x),\Theta(y)\rangle_\FQSym=0\\
&\Longleftrightarrow\forall y'\in \FQSym, \: \langle \Theta(x),y'\rangle=0\\
&\Longleftrightarrow\Theta(x) \in \FQSym^\perp\\
&\Longleftrightarrow \Theta(x)=0.
\end{align*}
So $\h_{\SP}^\perp=\Ker(\Theta)$. The proof is similar for $\h_{\HOP}$ and $\h_{\OF}$. \end{proof}

\subsection{Restriction to special plane posets}

\begin{prop}\label{21}
Let $n \in \mathbb{N}$. We partially order $\S_n$ by the right weak Bruhat order \cite{Stanley1}.
\begin{enumerate}
\item If $P \in \SPP(n)$, then $\displaystyle \Theta(P)=\sum_{\sigma \in \S_n,\: \sigma\leq \Phi_n(P)^{-1}} \sigma$.
\item Let $P\in \SP(n)$. There exists $\tau \in \S_n$, such that $S_P=\{\sigma \in \S_n\mid \sigma \leq \tau\}$ if, and only if, $P \in \SPP$.
\end{enumerate}\end{prop}

\begin{proof} $1$. We put $\tau=\Phi_n(P)^{-1}$. The aim is to prove that for all $\sigma \in \S_n$, $\sigma \in S_P$ if, and only if, $\sigma \leq \tau$.

Let us assume that $\sigma \in S_P$. We put:
\[I=\{(i,j)\:\mid\: i<_rj,\: \sigma^{-1}(i)<\sigma^{-1}(j)\}.\]
Let us prove that $\sigma \leq \tau$ by induction on $|I|$. If $|I|=0$, by definition of the elements of $S_P$, for all $i<j$:
\[i<_h j \Longleftrightarrow \sigma^{-1}(i)<\sigma^{-1}(j) \Longleftrightarrow \tau^{-1}(i)<\tau^{-1}(j).\]
So $\sigma=\tau$. Let us assume now that $|I|\geq 1$. Let us choose $(i,k)\in I$, such that $E=\sigma^{-1}(k)-\sigma^{-1}(i)$ is minimal.
If $E\geq 2$, let $j$ such that $\sigma^{-1}(i)<\sigma^{-1}(j) < \sigma^{-1}(k)$. Three cases are possible.
\begin{enumerate}
\item If $i<j<k$, by minimality of $E$, $i<_h j$ et $j<_h k$, so $i<_h k$. This contradicts $i<_r k$.
\item If $j<i<k$, by minimality of $E$, $j<_h k$. As $\sigma \in S_P$, $j <_r i$. As $i<_r k$, we obtain $j<_r k$. This contradicts $j <_h k$.
\item If $i<k<j$, by  minimality of $E$, $i<_h j$. As $\sigma \in S_P$, $k<_r j$. As $i<_r k$, $i<_r j$. This contradicts $i<_h j$.
\end{enumerate}
In all cases, this gives a contradiction. So $E=1$, that is to say $\sigma^{-1}(i)=\sigma^{-1}(k)-1$.
The permutation $\sigma'$ obtained from $\sigma$ by permuting $i$ and $k$ in the word representing $\sigma$ is greater than $\sigma$ for the right weak Bruhat order
by definition of this order; moreover, it is not difficult to show that it is also an element of $S_P$ (as $(i,k)\in I$), with a strictly smaller $|I|$.
By the induction hypothesis, $\sigma \leq \sigma'\leq \tau$. \\

Let us assume that $\sigma \leq \tau$ and let us prove that $\sigma \in S_P$. Then $\tau$ is obtained from $\sigma$ by a certain number $k$
of elementary transformations (that is to say the permutations of two adjacent letters $ij$ with $i<j$ in the word representing $\sigma$).
We proceed by induction on $k$. If $k=0$, then $\sigma=\tau$. If $k\geq 1$ there exists $\sigma' \in \S_n$, obtained from $\sigma$ by one elementary transformation,
such that $\tau$ is obtained from $\sigma'$ by $k-1$ elementary transformations. By the induction hypothesis, $\sigma' \in S_P$.
We put $\sigma=(\ldots a_i a_{i+1}\ldots)$, $\sigma'=(\ldots a_{i+1}a_i \ldots)$, with $a_i<a_{i+1}$. Let us prove that $\sigma\in S_P$. Let $k<_h l$. 
\begin{itemize}
\item If $k,l \neq a_i,a_{i+1}$, as $\sigma'\in S_P$, $\sigma^{-1}(k)=\sigma'^{-1}(k)<\sigma'^{-1}(l)=\sigma^{-1}(l)$.
\item If $k=a_i$, as $\sigma' \in S_P$, $l \neq a_{i+1}$. So $\sigma^{-1}(l)=\sigma'^{-1}(l)>\sigma'^{-1}(k)=\sigma^{-1}(k)+1$,
and $\sigma^{-1}(k)<\sigma^{-1}(l)$.
\item If $k=a_{i+1}$, then $l \neq a_i$ as $k<l$. So $\sigma^{-1}(l)\sigma'^{-1}(l)> \sigma'^{-1}(k)+1=\sigma^{-1}(k)$.
\item If $l=a_i$, then $k \neq a_{i+1}$ as $k<l$. Then $\sigma^{-1}(k)=\sigma'^{-1}(k)< \sigma'^{-1}(l)-1=\sigma^{-1}(l)$.
\item If $l=a_{i+1}$, as $\sigma \in S_P$, $k\neq a_i$. Then $\sigma^{-1}(k)=\sigma'^{-1}(k)<\sigma'^{-1}(l)=\sigma^{-1}(l)-1$,
and $\sigma^{-1}(k)<\sigma^{-1}(l)$.
\end{itemize}
Indeed, $\sigma \in S_P$.\\

$2. \Longleftarrow$. Comes from the first point, with $\tau=\Phi_n(P)^{-1}$.\\

$2.\Longrightarrow$. Let us assume that $S_P=\{\sigma \in \S_n\mid \sigma \leq \tau\}$ for a particular $\tau$. 
Then $Id_n \in S_P$, so $P$ is heap-ordered. \end{proof}

\begin{example}
Here is the Hasse graph of $\S_3$, partially ordered by the right weak Bruhat order:
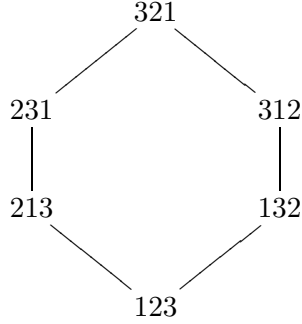
\[\xymatrix{&321\ar@{-}[rd]\ar@{-}[ld]&\\
231\ar@{-}[d]&&312\ar@{-}[d]\\
213&&132\\
&123\ar@{-}[ru]\ar@{-}[lu]&}\]
So:
\begin{align*}
\Theta(\tun\tun\tun)&=312+231+312+213+132+123\\
\Theta(\tun\tdeux)&=231+213+123\\
\Theta(\tdeux\tun)&=312+132+123\\
\Theta(\ptroisun)&=213+123\\
\Theta(\ttroisun)&=132+123\\
\Theta(\ttroisdeux)&=123.
\end{align*}\end{example}

As $\Phi_n:\SPP(n)\longrightarrow \S_n$ is a bijection:

\begin{cor}\label{22}
The restriction $\Theta_{\mid \h_{\SPP}}:\h_{\SPP}\longrightarrow \FQSym$ is an isomorphism.
\end{cor}

\begin{cor}\label{23}
The restriction of the pairing to $\h_{\SPP}$ is nondegenerate.
\end{cor}

\begin{proof} As the isomorphism $\Theta_{\mid \h_{\SPP}}$ is an isometry and the pairing of $\FQSym$ is nondegenerate. \end{proof}

\subsection{Restriction to heap-ordered forests}

\begin{notation}
Let $P=(P,\leq_1,\leq_2)$ be a special poset. If $i,j \in P$, we denote by $[i,j]_1$ the set of elements $k$ of $P$ such that $i\leq_1 k\leq_1 j$.
We denote by $R_P=\{(i,j) \in P^2\mid [i,j]_1=\{i,j\}, \: i\neq j\}$. This set is in fact the set of edges of the Hasse graph of $(P,\leq_1)$, so allows
to reconstruct the double poset $P$.
\end{notation}

\begin{prop} \label{24}
Let $P$ be a special poset with $n$ elements.
\begin{enumerate}
\item Let $i,j \in P$, such that $(j,i) \in R_P$. We define:
\begin{itemize}
\item $P_1\in \SP(n)$ such that $R_{P_1}=R_P\setminus\{(j,i)\}$;
\item $P_2\in \SP(n)$ such that $R_{P_2}=(R_P\setminus\{(j,i)\})\cup \{(i,j)\}$, after the elimination of redundant elements. 
\end{itemize}
Then $\Theta(P)=\Theta(P_1)-\Theta(P_2)$.
\item Let $i,j,k \in P$, all distinct, such that $(i,k)$ and $(j,k) \in R_P$. We define:
\begin{itemize}
\item $P_3\in \SP(n)$, such that $R_{P_3}=R_P\setminus \{(j,k)\}$;
\item $P_4 \in \SP(n)$, such that $R_{P_4}=(R_P\setminus \{(j,k)\})\cup \{(i,j)\}$, after the elimination of redundant elements;
\item $P_5 \in \SP(n)$, such that $R_{P_5}=(R_P\setminus \{(j,k),(i,k)\}) \cup \{(i,j),(j,k)\}$, after the elimination of redundant elements.
\end{itemize}
Then $\Theta(P)=\Theta(P_3)-\Theta(P_4)+\Theta(P_5)$.
\end{enumerate} \end{prop}

\begin{proof} $1$. We denote by $S$ the set of permutations $\sigma\in \S_n$ such that, for all $(x,y) \in R_P\setminus \{(i,j)\}$, $\sigma^{-1}(x)<\sigma^{-1}(y)$.
Then:
\begin{align*}
\Theta(P_1)&=\sum_{\sigma \in S} \sigma,&
\Theta(P)&=\sum_{\substack{\sigma \in S,\\ \sigma^{-1}(j)<\sigma^{-1}(i)}} \sigma,&
\Theta(P_2)&=\sum_{\substack{\sigma \in S,\\ \sigma^{-1}(j)>\sigma^{-1}(i)}} \sigma.
\end{align*}
As a consequence, $\Theta(P)+\Theta(P_2)=\Theta(P_1)$.\\

$2$. Note that $i$ and $j$ are not comparable for $\leq_1$ (otherwise, for example if $i<_1 j$, then $i<_1j<_1k$, and this contradicts the definition of $R_P$).
We denote by $S'$ the set of permutations $\sigma \in \S_n$, such that for all $(x,y) \in R_P\setminus \{(i,k),(j,k)\}$, $\sigma^{-1}(x)<\sigma^{-1}(y)$. Then:
\begin{align*}
\Theta(P)&=\sum_{\substack{\sigma \in S',\\ \sigma^{-1}(i),\sigma^{-1}(j)<\sigma^{-1}(k)}}\sigma,&
\Theta(P_3)&=\sum_{\substack{\sigma \in S',\\\sigma^{-1}(i)<\sigma^{-1}(k)}}\sigma,\\
\Theta(P_4)&=\sum_{\substack{\sigma \in S',\\ \sigma^{-1}(i)<\sigma^{-1}(j),\sigma^{-1}(k)}}\sigma,&
\Theta(P_5)&=\sum_{\substack{\sigma \in S',\\\sigma^{-1}(i)<\sigma^{-1}(j)<\sigma^{-1}(k)}}\sigma.
\end{align*}
We put:
\begin{align*}
S_1&=\sum_{\substack{\sigma \in S',\\ \sigma^{-1}(i)<\sigma^{-1}(j)<\sigma^{-1}(k)}}\sigma,&
S_2&=\sum_{\substack{\sigma \in S',\\\sigma^{-1}(j)<\sigma^{-1}(i)<\sigma^{-1}(k)}}\sigma,
\end{align*}
\[S_3=\sum_{\substack{\sigma \in S',\\ \sigma^{-1}(i)<\sigma^{-1}(k)<\sigma^{-1}(j)}}\sigma.\]

Then $\Theta(P)=S_1+S_2$, $\Theta(P_3)=S_1+S_2+S_3$, $\Theta(P_4)=S_1+S_3$ and $\Theta(P_5)=S_1$.
Hence, $\Theta(P)+\Theta(P_4)=\Theta(P_3)+\Theta(P_5)$.
\end{proof}

\begin{remark}
In other words, in the first case, one replaces a double subposet $\tddeux{$j$}{$i$}$ of $P$ by $\tdun{$i$}\tdun{$j$}-\tddeux{$i$}{$j$}$.
In the second case, one replaces a double subposet $\pdtroisun{$k$}{$i$}{$j$}$ by $\tddeux{$i$}{$k$} \tdun{$j$}
-\tdtroisun{$i$}{$k$}{$j$}+\tdtroisdeux{$i$}{$j$}{$k$}$.
\end{remark}

\begin{theo} \label{25}
Let $P \in \SP$. Applying repeatedly the two transformations of proposition \ref{24}, with $i<j$ in the first case, and $i<j<k$ in the second case, 
we can associate to $P$ a linear span of heap-ordered forests. This linear span does not depend on the way  the transformations are performed,
so is well-defined: we denote it by $\Upsilon(P)$. Then $\Upsilon$ defines a Hopf algebra morphism from $\h_{\SP}$ to $\h_{\HOF}$, 
such that the following diagram commutes:
\[\xymatrix{\h_{\SP}\ar[r]^{\Theta} \ar[d]_{\Upsilon}&\FQSym\\
\h_{\HOF}\ar[ru]_{\Theta}}\]
The restriction $\Theta_{\mid \h_{\HOF}}$ is an isomorphism, and $\Upsilon_{\mid \h_{\HOF}}=Id_{\h_{\HOF}}$.
Moreover, $\langle \Upsilon(x),\Upsilon(y)\rangle=\langle x,y\rangle$ for all $x,y\in \h_{\SP}$ (that is to say $\Upsilon$ respects the pairings).
\end{theo}

\begin{proof}  Let $P\in \SP$. It is clear that, using repeatedly the first transformation, we associate to $P$ a linear span of heap-ordered posets.
Then, using repeatedly the second transformation, we associate to this element of $\h_{\HOP}$ a linear span of heap-ordered forests.
Let $x$ be a linear span of heap-ordered forests obtained in this way. Using proposition \ref{24}, $\Theta(x)=\Theta(P)$.
As $\Theta:\h_{\SP}\longrightarrow \FQSym$ is surjective (as, for example, $\Theta_{\mid \h_{\SPP}}$ is an isomorphism),
$\Theta_{\mid \h_{\HOF}}$ is surjective. As $Card(\HOF(n))=Card(\S_n)=n!$ for all $n\in \mathbb{N}$, $\Theta_{\mid \h_{\HOF}}$ is bijective.
So $x$ is the unique antecedent of $\Theta(P) \in \FQSym$ in $\h_{\HOF}$, so $x=\left(\Theta_{\mid \h_{\HOF}}\right)^{-1}\circ \Theta(P)$
is unique, and $\Upsilon(P)=x$ is well-defined. Moreover, $\Upsilon=\left(\Theta_{\mid \h_{\HOF}}\right)^{-1}\circ \Theta$.
Consequently, it is a Hopf algebra morphism. As $\Theta$ respects the pairings, so does $\Upsilon$. \end{proof}

\begin{cor}\label{26}
\begin{enumerate}
\item $\Upsilon_{\mid \h_{\SPP}}:\h_{\SPP} \longrightarrow \h_{\HOF}$ is an isomorphism of graded Hopf algebras, and respects the pairings.
\item $\langle-,-\rangle_{\mid \h_{\HOF}}$ is nondegenerate.
\end{enumerate}\end{cor}

\begin{proof} By restriction in the commutative diagram of theorem \ref{25}, we obtain the following commutative diagram:
\[\xymatrix{\h_{\SPP}\ar[r]^{\Theta} \ar[d]_{\Upsilon}&\FQSym\\
\h_{\HOF}\ar[ru]_{\Theta}}\]
As the two restrictions of $\Theta$ are isomorphisms of graded Hopf algebras and respect the pairing, 
so is $\Upsilon_{\mid \h_{\SPP}}=(\Theta_{\mid \h_{\SPP}})^{-1}\circ \Theta_{\mid \FQSym}$. As $\Upsilon_{\mid \h_{\SPP}}$ is an isometry
and the pairing on $\h_{\SPP}$ is nondegenerate, the pairing on $\h_{\HOF}$ is nondegenerate. \end{proof}

\section{More algebraic structures on special posets}

\subsection{Recalls on $Dup$-$Dend$ bialgebras}
 
\label{sect51} Recall that a duplicial algebra \cite{Loday2} is a triple $(A,.,\nwarrow)$, where $A$ is a vector space, and $.,\nwarrow$ are two products on $A$, 
with the following axioms: for all $x,y,z \in A$,
\begin{equation}\label{E1}\left\{\begin{array}{rcl}
(xy)z&=&x(yz),\\
(x \nwarrow y)\nwarrow z&=&x \nwarrow (y \nwarrow z),\\
(xy)\nwarrow z&=&x (y \nwarrow z).
\end{array}\right. \end{equation}
In particular, the products $.$ and $\nwarrow$ are both associative.
A dendriform coalgebra (dual notion of dendriform algebra, \cite{Loday1,Ronco}) is a triple $(A,\Delta_\prec,\Delta_\succ)$, 
where $A$ is a vector space, and $\Delta_\prec$ and $\Delta_\succ$ are two coproducts on $A$, with the following axioms: for all $x\in A$,
\begin{equation} \label{E2} \left\{\begin{array}{rcl}
(\Delta_\prec\otimes Id)\circ \Delta_\prec(x)&=&(Id \otimes \tdelta)\circ \Delta_\prec(x),\\
(\Delta_\succ\otimes Id)\circ \Delta_\prec(x)&=&(Id \otimes \Delta_\prec)\circ \Delta_\succ(x),\\
(\tdelta\otimes Id)\circ \Delta_\succ(x)&=&(Id \otimes \Delta_\succ)\circ \Delta_\succ(x).
\end{array}\right. \end{equation}
Note that these axioms imply that $\tdelta=\Delta_\prec+\Delta_\succ$ is coassociative. We shall use the following Sweedler notations: for any $a\in A$,
\begin{align*}
\tdelta(a)&=a'\otimes a'',&\Delta_\prec(a)&=a'_\prec \otimes a''_\prec,& \Delta_\succ(a)&=a'_\succ \otimes a''_\succ.
\end{align*}
A $Dup$-$Dend$ bialgebra \cite{Foissy3} is a family $(A,.,\nwarrow,\Delta_\prec,\Delta_\succ)$, where $A$ is a vector space,
$.,\nwarrow:A\otimes A\longrightarrow A$ and $\Delta_\prec,\Delta_\succ:A \longrightarrow A \otimes A$, with the following properties:
\begin{itemize}
\item $(A,.,\nwarrow)$ is a duplicial algebra (axioms \ref{E1}).
\item $(A,\Delta_\prec,\Delta_\succ)$ is a dendriform coalgebra (axioms \ref{E2}).
\item For all $x,y\in A$:
\begin{equation}
\label{E3}\left\{\begin{array}{rcl}
\Delta_\prec(xy)&=&y \otimes x+y'_\prec \otimes x y''_\prec+xy'_\prec \otimes y''_\prec+x'y\otimes x''+x'y'_\prec \otimes x'' y''_\prec,\\[2mm]
\Delta_\succ(xy)&=&x \otimes y+xy'_\succ \otimes y''_\succ+y'_\succ \otimes x y''_\succ+x'\otimes x''y+x'y'_\succ \otimes x''y''_\succ;\\[2mm]
\Delta_\prec(x \nwarrow y)&=&x \nwarrow y'_\prec \otimes y''_\prec+x'_\prec \nwarrow y \otimes x''_\prec
+x'_\prec \nwarrow y'_\prec \otimes x''_\prec y''_\prec,\\[2mm]
\Delta_\succ(x \nwarrow y)&=&x \otimes y+x \nwarrow y'_\succ \otimes y''_\succ+x'_\succ \otimes x''_\succ \nwarrow y\\[1mm]
&&+x'_\prec \otimes x''_\prec y+x'_\prec \nwarrow y'_\succ \otimes x''_\prec y''_\succ.
\end{array}\right. \end{equation} \end{itemize}

\subsection{Another product on special posets}

\begin{defi}\textnormal{\begin{enumerate}
\item Let $P=(P,\leq_1,\leq_2)$ be a nonempty special poset.  The maximal element of $(P,\leq_2)$ will be denoted by $g_P$.
\item Let $P$ and $Q$ be two nonempty special poset. We define $P \nwarrow Q$ by:
\begin{itemize}
\item $P \nwarrow Q=P \sqcup Q$ as a set, and $P,Q$ are special subposets of $P \nwarrow Q$.
\item For all $x \in P$, $y\in Q$, $x \leq_2 y$.
\item For all $x\in P$, $y\in Q$, $x\leq_1 y$ if, and only if, $x \leq_1 g_P$.
\end{itemize}\end{enumerate}}\end{defi}

\begin{remark}
Let $P$ and $Q$ be two nonempty special posets. A Hasse graph of $P\nwarrow Q$ is obtained by grafting a Hasse graph of $Q$ on the vertex 
representing $g_P$ of a Hasse graph of $P$. For example:
\begin{align*}
\tdun{$1$}\tdun{$2$} \nwarrow \tddeux{$2$}{$1$}&=\tdun{$1$}\tdtroisdeux{$2$}{$4$}{$3$},&
\tddeux{$1$}{$2$} \nwarrow \tdun{$1$}\tdun{$2$}&=\tdquatrequatre{$1$}{$2$}{$4$}{$3$},&
\tddeux{$2$}{$1$}\nwarrow \tdun{$1$}\tdun{$2$}&=\tdquatreun{$2$}{$4$}{$3$}{$1$}
\end{align*}
\end{remark}

\begin{lemma}
($\h_{\SP}^+,.,\nwarrow)$ is a duplicial algebra.
\end{lemma}

\begin{proof} Let $P,Q,R$ be three nonempty special posets. The special posets $(P \nwarrow Q)\nwarrow R$ and $P \nwarrow (Q \nwarrow R)$ 
are both characterized by:
\begin{itemize}
\item $S=P\sqcup Q \sqcup R$ as a set, and $P,Q,R$ are special subposets of $S$.
\item For all $x\in P$, $y\in Q$, $z\in R$, $x\leq_2 y\leq_2 z$.
\item For all $x\in P$, $y\in Q$, $z\in R$, $x \leq_1 y$ if, and only if, $x \leq g_P$; $x \leq_1 z$ if, and only if, $x \leq_1 g_P$;
$y \leq_1 z$ if, and only if, $y\leq_1 g_Q$.
\end{itemize}
The last point comes from the fact that $g_{R \nwarrow S}=g_S$ for any nonempty special posets $R$ and $S$. So they are equal.\\

The special posets $(PQ)\nwarrow R$ and $P(Q \nwarrow R)$ are both characterized by:
\begin{itemize}
\item $S=P\sqcup Q \sqcup R$ as a set, and $P,Q,R$ are special subposets of $S$.
\item For all $x\in P$, $y\in Q$, $z\in R$, $x\leq_2 y\leq_2 z$.
\item For all $x\in P$, $y\in Q$, $z\in R$, $x$ and $y$ are not comparable for $\leq_1$; 
$x$ and $z$ are not comparable for $\leq_1$; $y \leq_1 z$ if, and only if, $y \leq_1 g_Q$.
\end{itemize}
So $\h_{\SP}^+$ is a duplicial algebra. \end{proof}

\begin{prop}
Let $P,Q$ be two nonempty special posets.
Then $P\nwarrow Q \in \HOP$ (respectively $\OF$, $\SPP$, $\HOF$, $\SPF$, $\SWNP$) if, and only if,
$P,Q \in \HOP$ (respectively $\OF$, $\SPP$, $\HOF$, $\SPF$, $\SWNP$).
\end{prop}

\begin{proof} We put $R=P \nwarrow Q$.\\

$\Longleftarrow$. In all the cases, this comes from the fact that $P$ and $Q$ are double subposets of $P\nwarrow Q$. \\

$\HOP. \Longrightarrow$. Recall from proposition \ref{9} that $R \in \HOP$ if, and only if, $R$ does not contain a double subposet isomorphic to $\tddeux{$2$}{$1$}$.
Let us assume that $P \nwarrow Q$ is not a heap-ordered poset. Then it contains two distinct elements $a,b$, such that $a\leq_1 b$ and $b\leq_2 a$.
If $a \in P$, then, by definition of $\leq_2$ on $R$, $b \in P$, so $P$ is not a heap-ordered poset.
If $a \in Q$, as $b\leq_1 a$, by definition of $\leq_1$ on $R$, $b \in Q$, so $Q$ is not a heap-ordered poset.\\

$\OF. \Longrightarrow$. Recall that $R$ is an ordered forest if, and only if, $(R,\leq_1)$ does not contain a double subposet isomorphic to $\ptroisun$ 
(see lemma 13 in \cite{Foissy1}). Let us assume that $R$ is not an ordered forest. Then it contains three different elements $a,b,c$, with $a\leq_2 b \leq_2 c$, 
such that one of the following assertions holds:
\begin{enumerate}
\item $b,c \leq_1 a$ and $b,c$ are not comparable for $\leq_1$: $(\{a,b,c\},\leq_1)=\pdtroisun{$a$}{$b$}{$c$}$.
\item $a,c \leq_1 b$ and $a,c$ are not comparable for $\leq_1$: $(\{a,b,c\},\leq_1)=\pdtroisun{$b$}{$a$}{$c$}$.
\item $a,b \leq_1 c$ and $a,b$ are not comparable for $\leq_1$: $(\{a,b,c\},\leq_1)=\pdtroisun{$c$}{$a$}{$b$}$.
\end{enumerate}
In the three cases, if the maximal element of $\{a,b,c\}$ for $\leq_1$ is in $P$, then, by definition of $\leq_1$ on $R$, $a,b,c \in P$, so $P$ is not an ordered forest.
Let us assume that this element is in $Q$.
In the first case, then, by definition of $\leq_2$ on $R$, $b,c \in Q$, so $Q$ is not an ordered forest. 
In the second case, we deduce similarly that $c \in Q$. If $a \in P$, then $a \leq_1 g_P$ in $P$ as $a \leq_1 b$ in $R$, so $a \leq_1 c$ in $R$: contradiction,
so $a \in Q$. As a consequence, $Q$ is not an ordered forest.
In the last case, then:
\begin{itemize}
\item If $a\in P$, $b\in Q$, then $a \leq_1 g_P$ in $P$ as $a\leq_1 c$ in $R$, so $a\leq_1 b$ in $R$: contradiction, this case is impossible.
\item Similarly, $a\in Q$, $b\in P$ is impossible.
\end{itemize}
So $a,b \in P$ or $a,b \in Q$. In the first subcase, $a,b \leq_1 g_P$ in $P$ as $a,b \leq_1 c$ in $R$, so $\{a,b,g_P\}$ is a subposet of $(P,\leq_1)$ 
isomorphic to $\ptroisun$: $P$ is not an ordered forest. In the second subcase, $Q$ contains $a,b,c$, so is not an ordered forest.\\

$\SPP.\Longrightarrow$. Recall from proposition \ref{9} that $R$ is a plane poset if, and only if, it is heap-ordered and does not contain a double subposet isomorphic to 
$\tddeux{$1$}{$3$}\tdun{$2$}$. Let us assume that $R$ is not a plane poset. If it is not heap-ordered, by the first point $P$ or $Q$ 
is not heap-ordered,  so is not a plane poset. Let us assume that there exists three different elements $a,b,c$ of $R$, such that $a\leq_2 b \leq_2 c$, $a\leq_1 c$, 
$a,b$ and $b,c$ are not comparable for $\leq_1$. By definition of $\leq_2$ on $R$, if $c \in P$, then $a,b \in P$, so $P \notin \SPP$.
If $c \in Q$ and $a\in Q$, then $b\in Q$ as $a\leq_2 b$, so $Q \notin \SPP$.
If $c\in Q$ and $a\in P$, then $a\leq_1 g_P$ in $P$. As $a$ and $b$ are not comparable for $\leq_1$ in $R$, $b \in P$.
As $b,c$ are not comparable for $\leq_1$ in $R$, $b$ and $g_P$ are not comparable for $\leq_1$ in $P$.
Let us consider $\{a,b,g_P\} \subseteq P$. By definition of $g_P$, $a\leq_ 2 b \leq_2 g_P$, so $\{a,b,g_P\}=\tddeux{$1$}{$3$}\tdun{$2$}$, so $P$ is not plane.\\

$\HOF.\Longrightarrow$. Comes from $\HOF=\OF \cap \HOP$.\\

$\SPF.\Longrightarrow$. Comes from $\SPF=\OF \cap \SPP$.\\ 

$\SWNP. \Longrightarrow$. Let us assume that $P\nwarrow Q$ is not a WN poset. If it is not plane, then 
by the third point, $P$ or $Q$ is not plane, so is not WN. Let us assume that $P\nwarrow Q$ is plane (so $P$ and $Q$ are plane).
Then $P\nwarrow Q$ contains a subposet $\{a,b,c,d\}$ isomorphic to $\pquatresix$ or $\pquatrecinq$. 
We assume that $a<_2 b<_2 c<_2 d$ in $P \nwarrow Q$. If $d \in P$,
then by definition of $P\nwarrow Q$, $\{a,b,c,d\} \subseteq P$, so $P$ is not WN. Similarly, if $a \in Q$, $Q$ is not WN.
We now assume that $a\in P$ and $d\in Q$. 
\begin{itemize}
\item If $\{a,b,c,d\}=\pquatresix$: as $a$ and $d$ are not comparable for $\leq_1$ in $P\nwarrow Q$, 
we do not have $a\leq_1 g_P$ in $P$. As $P$ is plane, it is heap-ordered, so $a$ and $g_P$ are not comparable for $\leq_1$ in $P$.
As $a<_1 c$ in $P\nwarrow Q$, necessarily $c \in P$. As $b<_2 c$ in $P\nwarrow Q$, $b\in P$. Moreover, as $b<_1 d$, $b<_1 g_P$.
As $c$ and $d$ are not comparable for $\leq_1$ in $P\nwarrow Q$, $c$ and $g_P$ are not comparable for $\leq_1$ in $P$.
So $\{a,b,c,g_P\}=\pquatresix$.
\item If $\{a,b,c,d\}=\pquatrecinq$: as $a<_1 d$ in $P\nwarrow Q$, $a<_1 g_P$ in $P$. As $a$ and $b$ are not comparable for $\leq_1$ in $P\nwarrow Q$,
necessarily $b\in P$. As $b<_1 d$, $b<_1 g_P$. As $c$ and $b$ are not comparable for $\leq_1$ in $P\nwarrow Q$, $c \in P$.
As $c$ and $d$ are not comparable for $\leq_1$, $c$ and $g_P$ are not comparable for $\leq_1$ in $P$. 
So $\{a,b,c,g_P\}=\pquatrecinq$.
\end{itemize}
In both cases, $P$ is not WN. \end{proof}

\begin{remark}
\begin{enumerate}
\item As a consequence, the augmentation ideals $\h_{\SP}^+$, $\h_{\HOP}^+$, $\h_{\SPP}^+$, $\h_{\OF}^+$, $\h_{\HOF}^+$, $\h_{\SWNP}^+$ 
and $\h_{\SPF}^+$  are duplicial algebras. 
\item It is proved in \cite{Foissy3} that $\h_{\SPF}^+$ is the free  duplicial algebra generated by $\tun$:
it is enough to observe that for any plane forest $F$, $g_F$ is the leaf of $F$ at most on the right, so $\nwarrow$, when restricted to plane forests, 
is precisely the product $\nwarrow$ defined in \cite{Foissy3}.
\end{enumerate}\end{remark}

\subsection{Dendriform coproducts on special posets}

For any nonempty special poset $P$, we put:
\begin{align*}
\Delta_\prec(P)&=\sum_{\substack{\mbox{\scriptsize $I$ non trivial ideal of $P$,}\\ g_P \notin I}} P\setminus I \otimes I,&
\Delta_\succ(P)=\sum_{\substack{\mbox{\scriptsize $I$ non trivial ideal of $P$,}\\ g_P \in I}} P\setminus I \otimes I.
\end{align*}
Note that $\Delta_\prec+\Delta_\succ=\tdelta$. Moreover, $\h_{\SP}^+$, $\h_{\HOP}^+$, $\h_{\SPP}^+$, $\h_{\OF}^+$, $\h_{\HOF}^+$, 
$\h_{\SWNP}^+$ and $\h_{\SPF}^+$ are stable under the coproducts $\Delta_\prec$ and $\Delta_\succ$.

\begin{prop} 
$\h_{\SP}^+$ is a $Dup$-$Dend$ bialgebra.
\end{prop}

\begin{proof} The proof is similar to the proof of proposition 20 in \cite{Foissy3}. Nevertheless, in order to help the reader, we give here a complete proof.
Let us first prove that ($\h_{\SP}^+,\Delta_\prec,\Delta_\succ)$ is a dendriform coalgebra.
It is enough to prove (\ref{E2}) if $x=P$ is a nonempty special poset. We put, as $\tdelta$ is coassociative,
$(\tdelta \otimes Id) \circ \tdelta(P)=(Id \otimes \tdelta)\circ \tdelta(P)=\sum P^{(1)} \otimes P^{(2)} \otimes P^{(3)}$,
where $P^{(1)},P^{(2)},P^{(3)}$ are subposets of $P$. Then:
\[\left\{\begin{array}{rcccl}
(\Delta_\prec\otimes Id)\circ \Delta_\prec(P)&=&(Id \otimes \tdelta)\circ \Delta_\prec(P)
&=&\displaystyle \sum_{g_P \in P^{(1)}} P^{(1)} \otimes P^{(2)}\otimes P^{(3)},\\
(\Delta_\succ\otimes Id)\circ \Delta_\prec(P)&=&(Id \otimes \Delta_\prec)\circ \Delta_\succ(P)
&=&\displaystyle \sum_{g_P \in P^{(2)}} P^{(1)} \otimes P^{(2)}\otimes P^{(3)},\\
(\tdelta\otimes Id)\circ \Delta_\succ(P)&=&(Id \otimes \Delta_\succ)\circ \Delta_\succ(P)
&=&\displaystyle \sum_{g_P \in P^{(3)}} P^{(1)} \otimes P^{(2)}\otimes P^{(3)}.
\end{array}\right.\]
So $\h_{\SP}^+$ is a dendriform coalgebra. \\

Let us now prove axioms (\ref{E3}). It is enough prove these formulas if $x=P$, $y=Q$ are nonempty plane forests. 
Let $I$ be a non trivial ideal of $PQ$ or $P\nwarrow Q$. We put $I'=I \cap P$ and $I''=I \cap Q$. 
As $I$ is non trivial, $I'$ and $I''$ are not simultaneously empty and not simultaneously total.\\

Let us first compute $\Delta_\prec(PQ)$. We have to consider non trivial ideals $I$ of $PQ$, such that $g_{PQ} \notin I$.
As $g_{PQ}=g_Q$, $I''\neq Q$. So five case are possible.
\begin{itemize}
\item $I'=P$, $I''=\emptyset$: this gives the term $Q \otimes P$.
\item $I'=P$, $I''\neq \emptyset, Q$: this gives the term $Q'_\prec \otimes PQ''_\prec$.
\item $I'=\emptyset$, $I'' \neq \emptyset,Q$: this gives the term $PQ'_\prec \otimes PQ''_\prec$.
\item $I'\neq \emptyset,P$, $I''=\emptyset$: this gives the term $P'Q\otimes P''$.
\item $I'\neq \emptyset,P$, $I''\neq \emptyset,Q$: this gives the term $P'Q'_\prec \otimes P'' Q''_\prec$.
\end{itemize}

Let us compute $\Delta_\succ(PQ)$. We  have to consider non trivial ideals $I$ of $PQ$, such that $g_{PQ} \in I$.
As $g_{PQ}=g_Q$, $I''\neq \emptyset$. So five cases are possible:
\begin{itemize}
\item $I'=\emptyset$, $I''=Q$: this gives the term $P \otimes Q$.
\item $I'=\emptyset$, $I''\neq \emptyset, Q$: this gives the term $PQ'_\succ \otimes Q''_\succ$.
\item $I'=P$, $I'' \neq \emptyset,Q$; this gives the term $Q'_\succ \otimes PQ''_\succ$.
\item $I'\neq \emptyset,P$, $I''=Q$: this gives the term $P'\otimes P''Q$.
\item $I'\neq \emptyset,P$, $I''\neq \emptyset,Q$: this gives the term $P'Q'_\succ \otimes P'' Q''_\succ$.
\end{itemize}

We now compute $\Delta_\prec(P\nwarrow Q)$. We have to consider non trivial ideals $I$ of $P\nwarrow Q$, such that $g_{P\nwarrow Q} \notin I$. 
As $g_{P\nwarrow Q}=g_Q$, $I''\neq Q$. Moreover, if $g_P \in I$, then, as $I$ is an ideal, $Q \subseteq I$ so $I''=Q$: impossible. So $g_P \notin I'$. 
So three cases are possible.
\begin{itemize}
\item $I'=\emptyset$, $I'' \neq \emptyset,Q$; this gives the term $P\nwarrow Q'_\prec \otimes PQ''_\prec$.
\item $I'\neq \emptyset,P$, $I''=\emptyset$: this gives the term $P'_\prec \nwarrow Q\otimes P''_\nwarrow$.
\item $I'\neq \emptyset,P$, $I''\neq \emptyset,Q$: this gives the term $P'_\prec \nwarrow Q'_\prec \otimes P''_\prec Q''_\prec$.
\end{itemize}

Finally, let us compute $\Delta_\succ(P\nwarrow Q)$. We  have to consider non trivial ideals $I$ of $P\nwarrow Q$, such that $g_{P\nwarrow Q} \in I$.
As $g_{P\nwarrow Q}=g_Q$, $I''\neq \emptyset$. Moreover, if $g_P \in I'$, as $I$ is an ideal, $I''=Q$. As $I'$ and $I''$ are not simultaneously total,
this implies that $I' \neq P$. So five cases are possible:
\begin{itemize}
\item $I'=\emptyset$, $I''=Q$: this gives the term $P \otimes Q$.
\item $I'=\emptyset$, $I''\neq \emptyset, Q$: this gives the term $P\nwarrow Q'_\succ \otimes Q''_\succ$.
\item $I'\neq \emptyset,P$, $g_P \in I'$: this gives the term $P'_\succ \otimes P''_\succ\nwarrow Q$.
\item $I'\neq \emptyset,P$, $g_P \notin I'$, $I''=Q$: this gives the term $P'_\prec \otimes P''_\prec Q$.
\item $I'\neq \emptyset,P$, $g_P \notin I'$,  $I''\neq \emptyset,Q$: this gives the term $P'_\prec \nwarrow Q'_\succ \otimes P''_\prec Q''_\succ$.
\end{itemize}
So $\h_\SP^+$ is a $Dup$-$Dend$ bialgebra. \end{proof}

\begin{remark}\begin{enumerate}
\item As a consequence, the augmentation ideals $\h_{\SP}^+$, $\h_{\HOP}^+$, $\h_{\SPP}^+$, $\h_{\OF}^+$, $\h_{\HOF}^+$, $\h_{\SWNP}^+$ 
and $\h_{\SPF}^+$ are $Dup$-$Dend$ bialgebras. 
\item The rigidity theorem of  \cite{Foissy3} implies that $\h_{\SP}$, $\h_{\HOP}$, $\h_{\SPP}$, $\h_{\OF}$, $\h_{\HOF}$, $\h_{\SWNP}$ and $\h_{\SPF}$
are isomorphic to non commutative Connes-Kreimer Hopf algebras of decorated plane trees, with particular graded sets of decorations. 
The cardinal of the components of these graded sets can be computed by manipulations of formal series. For example:
\[\begin{array}{|c|c|c|c|c|c|c|c|c|}
\hline n&1&2&3&4&5&6&7&8\\
\hline|\D_{\SP}(n)|&1&1&10&148&3\:336&112\:376&5\:591\:196&406\:621\:996\\
\hline|\D_{\OF}(n)|&1&1&7&66&786&11\:278&189\:391&3\:648\:711\\
\hline|\D_{\HOF}(n)|=|\D_{\SPF}(n)|&1&0&1&6&39&284&2\:305&20\:682\\
\hline|\D_{\SWNP}(n)|&1&0&1&4&17&76&353&1\:688\\
\hline|\D_{\SPF}(n)|&1&0&0&0&0&0&0&0\\
\hline \end{array}\]
We obtain sequences A122705 for $\D_{\OF}$ and A122827 for $\D_{\HOF}$ in \cite{Sloane}.
\end{enumerate}\end{remark}

\subsection{Application to $\FQSym$}

Let $\sigma\in \S_n$ be a permutation ($n\geq 1$). We put:
\begin{align*}
\Delta_\prec(\sigma)&=\sum_{k=\sigma^{-1}(n)}^{n-1} \sigma^{(k)}_1 \otimes \sigma^{(k)}_2, &
\Delta_\succ(\sigma)&=\sum_{k=1}^{\sigma^{-1}(n)-1} \sigma^{(k)}_1 \otimes \sigma^{(k)}_2.
\end{align*}
Remark that $\Delta_\prec+\Delta_\succ=\tdelta$. \\

\begin{example}
\begin{align*}
\Delta_\prec((12543))&=(123) \otimes (21)+(1243) \otimes (1),&
\Delta_\succ((12543))&=(1)\otimes (1432)+(12)\otimes (321).
\end{align*}
\end{example}

Let $\sigma,\tau$ be two permutations of respective degrees $k$ and $l$, with $k,l\geq 1$. We put:
\[\sigma \nwarrow \tau=\sum_{\substack{\zeta \in Sh(k,l)\\ \zeta(k+1) \geq \zeta(\sigma^{-1}(k))}} (\sigma \otimes \tau) \circ \zeta^{-1}.\]
In other terms, $\sigma \nwarrow \tau$ is the sum of the shufflings of the word representing $\sigma$ and the word representing $\tau$ shifted by $k$,
such that the letters of $\tau$ are all after the greatest letter of $\sigma$. In particular, if $\sigma^{-1}(k)=k$, 
then $\sigma \nwarrow \tau=\sigma \otimes \tau$. \\

\begin{example}
\begin{align*}
123\nwarrow 12&=12345,\\
132\nwarrow 12&=13245+13425+13452,\\
312\nwarrow 12&=31245+31425+34125+34152+34512.
\end{align*}
\end{example}

\begin{prop} 
These products and coproducts make $\FQSym^+$ a $Dup$-$Dend$ bialgebra. Moreover, $\Theta:\h_{\SP}^+\longrightarrow \FQSym^+$ 
is a morphism of $Dup$-$Dend$ bialgebras.
\end{prop}

\begin{proof} We first prove the compatibility of $\Theta$ with $\nwarrow$. Let $P$ and $Q$ be two nonempty special posets, 
of respective degrees $k$ and $l$. We first show:
\[S_{P \nwarrow Q}=\bigsqcup_{\sigma \in S_P,\:\tau \in S_Q} \bigsqcup_{\substack{\zeta \in Sh(k,l)\\ \zeta(k+1)\geq \zeta(\sigma^{-1}(k))}}
\{(\sigma \otimes \tau) \circ \zeta^{-1}\}.\]

$\subseteq$. Let $\chi  \in S_{P \nwarrow Q}$. There exists a unique $(\sigma,\tau,\zeta)\in \Sigma_k \times \Sigma_l\times Sh(k,l)$, such that 
$\chi=(\sigma \otimes \tau) \circ \zeta^{-1}$. Let us prove that $\sigma \in S_P$. If $i >_1 j$ in $P$,
then $i>_1 j$ in $P \nwarrow Q$, so:
\begin{align*}
\chi^{-1}(i)&\geq \chi^{-1}(j),\\
\zeta \circ (\sigma^{-1} \otimes \tau^{-1})(i) &\geq \zeta \circ (\sigma^{-1} \otimes \tau^{-1})(j),\\
\zeta \circ \sigma^{-1}(i) &\geq\zeta\circ \sigma^{-1}(j),\\
\sigma^{-1}(i)&\geq\sigma^{-1}(j),
\end{align*}
as $\zeta$ is increasing on $\{1,\ldots,k\}$. So $\sigma \in S_P$. Similarly, $\tau \in S_Q$. Moreover, 
the element $\tau(1)+k$ belongs to $Q$ in $P \nwarrow Q$, so $\tau(1)+k >_1 k$ in $P\nwarrow Q$. As a consequence:
\begin{align*}
\chi^{-1}(\tau(1)+k)&\geq \chi^{-1}(k),\\
\zeta \circ (\sigma^{-1} \otimes \tau^{-1})(\tau(1)+k) &\geq \zeta \circ (\sigma^{-1} \otimes \tau^{-1})(k),\\
\zeta(k+1)&\geq\zeta\circ \sigma^{-1}(k).
\end{align*}

$\supseteq$. Let $\sigma \in S_P$, $\tau \in S_Q$ and $\zeta \in Sh(k,l)$, such that $\zeta(k+1)\geq \zeta(\sigma^{-1}(k))$. We put
$\chi=(\sigma \otimes \tau) \circ \zeta^{-1}$. Let $i,j$ be two elements of $P \nwarrow Q$, such that $i>_1 j$. Three cases can occur:
\begin{itemize}
\item $i,j$ are elements of $P$. Then $\sigma^{-1}(i)\geq \sigma^{-1}(j)$, so $(\sigma^{-1}\otimes \tau^{-1})(i) \geq (\sigma^{-1}\otimes \tau^{-1})(j) $, 
and finally $\sigma^{-1}(i)=\zeta \circ(\sigma^{-1}\otimes \tau^{-1})(i) \geq \zeta \circ (\sigma^{-1}\otimes \tau^{-1})(j)=\sigma^{-1}(j)$.
\item $i,j$ are elements of $Q$. The same proof holds.
\item $i$ is an element of $Q$ and $j$ is an element of $P$. Then $i >_1 k$ in $P \nwarrow Q$.
By definition of $P \nwarrow Q$,  $k >_1 j$ in $P$, so by the first point $\sigma^{-1}(k)\geq \sigma^{-1}(j)$.

Moreover, $i+1 \geq k+1$, so $\sigma^{-1}(i) \geq \zeta(k+1)$ as $\zeta$ is increasing on $\{k+1,\ldots,k+l\}$. Then:
\[\sigma^{-1}(i) \geq \zeta(k+1)\geq \zeta(\sigma^{-1}(k))=\sigma^{-1}(k) \geq \sigma^{-1}(j).\]
\end{itemize}
Finally, for any nonempty special posets $P$ and $Q$ of respective degrees $k$ and $l$:
\[\Theta(\P \nwarrow Q)=\sum_{\sigma \in S_P,\:\tau \in S_Q} \sum_{\substack{\zeta \in Sh(k,l)\\ \zeta(k+1)\geq \zeta(\sigma^{-1}(k))}}
(\sigma \otimes \tau) \circ \zeta^{-1}=\sum_{\sigma \in S_P,\:\tau \in S_Q} \sigma \nwarrow \tau=\Theta(P) \nwarrow \Theta(Q).\]

We now prove the compatibility of $\Theta$ and the two coproducts $\Delta_\prec$ and $\Delta_\succ$.
Let $P \in \SP(n)$. As $\Theta$ is a morphism of Hopf algebras, there exists a bijection:
\[\left\{\begin{array}{rcl}
S_P\times \{1,\ldots,n-1\}&\longmapsto&\displaystyle \bigsqcup_{\mbox{\scriptsize $I$ non trivial ideal of $P$}} S_{P\setminus I} \times S_I\\
(\sigma,k)&\longmapsto&\left(\sigma_1^{(k)},\sigma_2^{(k)}\right),
\end{array}\right.\]
where this pair belongs to the term of the union indexed by $I=\{\sigma(k+1),\ldots,\sigma(n)\}$.
So, if $(\sigma,k) \in S_F\times \{1,\ldots,n-1\}$, $k \geq \sigma^{-1}(n)$ if, and only if, $n=g_P$ is not an element of $I$. So:
\[(\Theta \otimes \Theta) \circ \Delta_\prec(F)=\sum_{g_P \notin I} P\setminus I \otimes I
\sum_{\substack{\sigma \in S_{P\setminus I}\\ \tau \in S_I}} \sigma \otimes \tau
=\sum_{\sigma \in S_P} \sum_{k=\sigma^{-1}(n)}^{n-1} \sigma_1^{(k)} \otimes \sigma_2^{(k)}
=\sum_{\sigma \in S_P} \Delta_\prec(\sigma)=\tdelta \circ \Theta(F).\]
Similarly, $(\Theta \otimes \Theta) \circ \Delta_\succ=\Delta_\succ \circ \Theta$. \\

As $\Theta_{\mid \h_{\HOF}}\longrightarrow \FQSym$ is an isomorphism and $\h_{\HOF}^+$ is a $Dup$-$Dend$ bialgebra,
$\FQSym^+$ is also a $Dup$-$Dend$ bialgebra. \end{proof}

\begin{remark} \begin{enumerate}
\item It is of course possible to prove directly that $\FQSym^+$ a $Dup$-$Dend$ bialgebra.
\item A similar structure of $Dup$-$Dend$ bialgebra structure exists on the Hopf algebra of parking functions $\PQSym$ \cite{Novelli}, 
replacing, for a parking function $\sigma$, $\sigma^{-1}(n)$ by the maximal integer $i$ such that $\sigma(i)$ is maximal.
\end{enumerate}\end{remark}

\section{Dendriform structures on special plane forests}

The aim of this section is to prove that the restriction of the pairing to $\h_{\SPF}$ is nondegenerate (corollary \ref{37}). We first recall the classical result:

\begin{lemma}
The restriction of $\langle-,-\rangle$ to $\K[\tun]$ is nondegenerate if, and only if, the characteristic of $\K$ is zero.
\end{lemma}

\begin{proof} As the homogeneous components of $\K[\tun]$ are one-dimensional, this restriction is nondegenerate if, and only if,
$\langle \tun^n,\tun^n\rangle$ is a non-zero element of $\K$ for all $n\in \mathbb{N}$. 
Moreover, it is not difficult to show that $\langle\tun^n,\tun^n\rangle=n!$. \end{proof}

\subsection{Dendriform coproducts}

\begin{notation}
Let $P$ be a plane poset, seen as a special poset. The smallest element for the total order of $P$ will be denoted by $s_P$.
\end{notation}

\begin{prop}
For any nonempty plane poset $P$, we put:
\begin{align*}
\Delta'_\prec(P)&=\sum_{\substack{\mbox{\scriptsize $I$ non trivial ideal of $P$}\\ s_P \notin I}} P\setminus I \otimes I,&
\Delta'_\succ(P)&=\sum_{\substack{\mbox{\scriptsize $I$ non trivial ideal of $P$}\\ s_P \in I}} P\setminus I \otimes I.
\end{align*}
Then $(\h_{\SPP}^+,\Delta'_\prec,\Delta'_\succ)$ is a dendriform coalgebra. Moreover, for all $x,y\in \h_{\SPP}^+$:
\begin{align}
\label{E4} \Delta'_\prec(xy)&=x \otimes y+x'_\prec y\otimes y''_\prec+x'_\prec \otimes x''_\prec y+xy'\otimes y''+x'_\prec y'\otimes x''_\prec y'',\\
\label{E5} \Delta'_\succ(xy)&=y\otimes x+x'_\succ y\otimes x''_\succ+x'_\succ \otimes x''_\succ y+y'\otimes xy''+x'_\succ y'\otimes x''_\succ y''.
\end{align}\end{prop}

\begin{proof} Let us first prove the (\ref{E2}) for all $x\in \h_{\SPP}^+$.
It is enough to prove this if $x=P$ is a nonempty special poset. We put, as $\tdelta$ is coassociative,
$(\tdelta \otimes Id) \circ \tdelta(P)=(Id \otimes \tdelta)\circ \tdelta(P)=\sum P^{(1)} \otimes P^{(2)} \otimes P^{(3)}$,
where $P^{(1)},P^{(2)},P^{(3)}$ are subposets of $P$. Then:
\[\left\{\begin{array}{rcccl}
(\Delta'_\prec\otimes Id)\circ \Delta'_\prec(P)&=&(Id \otimes \tdelta)\circ \Delta'_\prec(P)
&=&\displaystyle \sum_{s_P \in P^{(1)}} P^{(1)} \otimes P^{(2)}\otimes P^{(3)},\\
(\Delta'_\succ\otimes Id)\circ \Delta'_\prec(P)&=&(Id \otimes \Delta'_\prec)\circ \Delta'_\succ(P)
&=&\displaystyle \sum_{s_P \in P^{(2)}} P^{(1)} \otimes P^{(2)}\otimes P^{(3)},\\
(\tdelta\otimes Id)\circ \Delta'_\succ(P)&=&(Id \otimes \Delta'_\succ)\circ \Delta'_\succ(P)
&=&\displaystyle \sum_{s_P \in P^{(3)}} P^{(1)} \otimes P^{(2)}\otimes P^{(3)}.
\end{array}\right.\]
So $\h_{\SP}^+$ is a dendriform coalgebra.\\

 It is enough prove formulas (\ref{E4}) and (\ref{E5}) if $x=P$, $y=Q$ are nonempty plane forests.  Let $I$ be a non trivial ideal of $PQ$. 
We put $I'=I \cap P$ and $I''=I \cap Q$.  As $I$ is non trivial, $I'$ and $I''$ are not simultaneously empty and not simultaneously total.\\

Let us first compute $\Delta'_\prec(PQ)$. We have to consider non trivial ideals $I$ of $PQ$, such that $s_{PQ} \notin I$.
As $s_{PQ}=s_P$, $I'\neq P$. So five case are possible.
\begin{itemize}
\item $I'=\emptyset$, $I''=Q$: this gives the term $P \otimes Q$.
\item $I'=\emptyset$, $I''\neq \emptyset, Q$: this gives the term $PQ' \otimes Q''$.
\item $I'\neq \emptyset,P$, $I''=\emptyset$: this gives the term $P'_\prec Q\otimes P''_\prec$.
\item $I'\neq \emptyset,P$, $I''=Q$: this gives the term $P'_\prec\otimes P''_\prec Q$.
\item $I'\neq \emptyset,P$, $I''\neq \emptyset,Q$: this gives the term $P'_\prec Q'\otimes P''_\prec Q''$.
\end{itemize}
The proof of formula (\ref{E5}) is similar. \end{proof}

\begin{remark}\begin{enumerate}
\item In other words, $(\h_{\SPP}^+,.^{op},(\Delta'_\succ)^{op},(\Delta'_\prec)^{op})$ is a codendriform bialgebra in the sense of \cite{Foissy4}.
\item $\h_{\SPF}^+$ is clearly stable under both coproducts $\Delta'_\prec$ et $\Delta'_\succ$, so $(\h_{\SPF}^+,.^{op},(\Delta'_\succ)^{op},(\Delta'_\prec)^{op})$
is a codendriform subcoalgebra of $\h_{\SPP}^+$.
\end{enumerate}\end{remark}

\subsection{Dendriform products on special plane forests}

From \cite{Foissy5}, $\h_{\SPF}^+$ is the free dendriform algebra generated by $\tun$. Moreover, for all nonempty plane forest $F$,
$\tun \prec F=B^+(F)$, the rooted tree obtained by grafting the roots of $F$ on a common root. It is also proved that $(\h_{\SPF}^+,\prec,\succ,\tdelta^{op})$
is a dendriform Hopf algebra \cite{LR2}, so, for all $x,y\in \h_{\SPF}^+$:
\begin{align}
\label{E6} \tdelta(x\prec y)&=x \otimes y+x\prec y'\otimes y''+x'\otimes x''y+x'\prec y\otimes x''+x'\prec y'\otimes x''y'',\\
\label{E7} \tdelta(x\succ y)&=y\otimes x+x\succ y'\otimes y''+y'\otimes xy''+x'\succ y\otimes x''+x'\succ y'\otimes x''y''.
\end{align}

\begin{prop}
For all $x,y\in \h_{\SPF}^+$:
\begin{align}
\label{E8} \Delta'_\prec(x\prec y)&=x\otimes y+x\prec y'\otimes y''+x'_\prec \otimes x''_\prec y+x'_\prec \prec y\otimes x''_\prec+x'_\prec \prec y'\otimes x''_\prec y'',\\
\label{E9} \Delta'_\succ(x\prec y)&=x'_\succ \otimes x''_\succ y+x'_\succ \prec y\otimes x''_\succ+x'_\succ \prec y'\otimes x''_\succ y'',\\
\label{E10} \Delta'_\prec(x\succ y)&=x'_\prec \succ y\otimes x''_\prec+x \succ y'\otimes y''+x'_\prec \succ y'\otimes x''_\prec y'',\\
\label{E11} \Delta'_\succ(x\succ y)&=y\otimes x+y'\otimes xy''+x'_\succ \succ y\otimes x''_\succ+x'_\succ \succ y'\otimes x''_\succ y''.
\end{align}\end{prop}

\begin{proof} For fixed $x,y$, note that $(\ref{E8})+(\ref{E10})=(\ref{E4})$, $(\ref{E9})+(\ref{E11})=(\ref{E5})$, 
$(\ref{E8})+(\ref{E9})=(\ref{E6})$, and $(\ref{E10})+(\ref{E11})=(\ref{E7})$. As a consequence, for fixed $x,y$,
(\ref{E8}), (\ref{E9}), (\ref{E10}) and (\ref{E11}) are equivalent.

We now prove (\ref{E8})-(\ref{E11}) for $x,y$ two non empty plane forest, by induction on the degree $n$ of $x$.
If $n=1$, then $x=\tun$. Then:
\[\Delta'_\prec(x\prec y)=\tun \otimes y+B^+(y') \otimes y''=x \otimes y+x\prec y'\otimes y''.\]
So (\ref{E8}) (hence, (\ref{E9})-(\ref{E11})) holds for $x=\tun$, as $\Delta'_\prec(x)=0$. Let us assume the result at all rank $<n$. Two subcases occur.
\begin{itemize}
\item The plane forest $x$ is a tree. Then there exists $x_1$ of degree $n-1$, such that $x=B^+(x_1)=x\prec x_1$.
So $x\prec y=(\tun \prec x_1)\prec y=\tun \prec (x_1y)$. So:
\begin{align*}
\Delta'_\prec(x\prec y)&=\Delta'_\prec (\tun \prec (x_1y))\\
&=\tun \otimes (x_1y)+B^+((x_1y)')\otimes (x_1y)''\\
&=\tun \otimes (x_1y)+\tun \prec x_1\otimes y+\tun \prec y\otimes x_1+\tun \prec (x_1'y)\otimes x_1''+\tun \prec x_1'\otimes x_1''y\\
&+\tun \prec (x_1y')\otimes y''+\tun \prec y'\otimes x_1y''+\tun \prec (x_1'y')\otimes x_1''y''\\
&=(\tun \prec x_1\otimes y)+(\tun \prec (x_1y')\otimes y'')+(\tun \otimes (x_1y)+\tun \prec x_1'\otimes x_1''y)\\
&+(\tun \prec y\otimes x_1+\tun \prec (x_1'y)\otimes x_1'')+(\tun \prec y'\otimes x_1y''+\tun \prec(x_1'y')\otimes x_1''y'')\\
&=x \otimes y+x\prec y'\otimes y''+x'_\prec \otimes x''_1y+x'_\prec\prec y\otimes x''_\prec+x'_\prec \prec y'\otimes x''_\prec y''.
\end{align*}
\item The plane forest $x$ is not a tree. Then it can be written as $x=x_1x_2$, such that the induction hypothesis holds for $x_1$ and $x_2$. Hence:
\[x\prec y=(x_1 \prec x_2)\prec y+(x_1 \succ x_2)\prec y=x_1 \prec (x_2y)+x_1 \succ (x_2 \prec y).\]
Applying (\ref{E8}) and (\ref{E10}) for $x_1$ (induction hypothesis), then (\ref{E4}) for $x_2$, then arranging the terms, gives (\ref{E8}) for $x$.
\end{itemize}
So the induction hypothesis holds for $x$ in both cases. \end{proof}

\begin{remark}
In other words, $(\h_{\SPF}^+,\succ^{op},\prec^{op},(\Delta'_\succ)^{op},(\Delta'_\prec)^{op})$ is a bidendriform bialgebra 
in the sense of \cite{Foissy6}. By the bidendriform rigidity theorem, it is a free dendriform algebra, and a cofree dendriform coalgebra. As a direct consequence:
\end{remark}

\begin{lemma} \label{35}
As a dendriform algebra, $\h_{\SPF}^+$ is freely generated by $\tun$. Moreover, the space $Prim_{tot}(\h_{\SPF}^+)
=\Ker(\Delta'_\prec)\cap \Ker(\Delta'_\succ)$ is one-dimensional, generated by $\tun$.
\end{lemma}

\begin{lemma} \label{36}
For all $x,y,z\in \h_{\SPF}^+$:
\[\langle x\prec y,z\rangle=\langle x\otimes y,\Delta'_\prec(y)\rangle \mbox{ and }
\langle x\succ y,z\rangle=\langle x\otimes y,\Delta'_\succ(y)\rangle.\]
\end{lemma}

\begin{proof} As $\langle-,-\rangle$ is a Hopf pairing, it is enough to prove one of these two formulas.
Moreover, it is enough to prove it for $x,y,z$ three non empty plane forests. 
We prove the first one, by induction on the degree $n$ of $x$. If $n=1$, then $x=\tun$ and $x\prec y=B^+(y)$.
Let $\sigma\in S(B^+(y),z)$. As $1$ is the root of $B^+(y)$, for all $j$, $1\leq_h j$ in $B^+(y)$. As $\sigma\in S(B^+(y),z)$,
$\sigma(1)\leq \sigma(i)$ for all $i$, so $\sigma(1)=1$. Let us denote by $z_1$ the plane forest obtained by deleting the vertex $1$ of $z$;
then $S(B^+(y),z)$ is in bijection by $S(y,z_1)$. Moreover, by definition of $\Delta'_\prec$:
\[\Delta'_\prec(z)=\tun \otimes z_1+\mbox{terms $z'\otimes z''$, $z'$ homogeneous of degree $\geq 2$}.\]
So, by homogeneity of the pairing:
\[\langle x \otimes y,\Delta'_\prec(z)\rangle=\langle \tun,\tun \rangle \langle y,z_1\rangle+0=|S(y,z_1)|=|S(B^+(y),z)|=\langle x\prec y,z\rangle.\]

Let us assume the result at all rank $<n$. Two subcases occur.
\begin{itemize}
\item The plane forest $x$ is a tree. Let us put $x=B^+(x_1)=\tun \prec x_1$. Using the result at rank $1$:
\begin{align*}
\langle x\prec y,z\rangle&=\langle \tun \prec (x_1y),z\rangle\\
&=\langle \tun \otimes x_1y,\Delta'_\prec(z)\rangle\\
&=\langle \tun \otimes x_1\otimes y,(Id \otimes \tdelta)\circ \Delta'_\prec (z)\rangle\\
&=\langle \tun \otimes x_1\otimes y,(\Delta'_\prec \otimes Id)\circ \Delta'_\prec (z)\rangle\\
&=\langle \tun \prec x_1, \Delta'_\prec(z)\rangle.
\end{align*}
\item The plane forest $x$ is not a tree. Then it can be written as $x=x_1x_2$, such that the induction hypothesis holds for $x_1$ and $x_2$. Hence:
\begin{align*}
\langle (x_1x_2)\prec y,z\rangle&=\langle x_1 \prec (x_2y),z\rangle+\langle x_1 \succ(x_2 \prec y),z\rangle\\
&=\langle x_1\otimes x_2\otimes y,(Id \otimes \tdelta)\circ \Delta'_\prec(z)\rangle
+\langle x_1\otimes x_2\otimes y,(Id \otimes \Delta'_\prec)\circ \Delta'_\succ(z)\rangle\\
&=\langle x_1\otimes x_2\otimes y,(\Delta'_\prec \otimes Id)\circ \Delta'_\prec(z)\rangle
+\langle x_1\otimes x_2\otimes y,(\Delta'_\succ\otimes Id)\circ \Delta'_\prec(z)\rangle\\
&=\langle x_1 \prec x_2 \otimes y,\Delta'_\prec(z)\rangle+\langle x_1 \succ x_2 \otimes y,\Delta'_\prec(z)\rangle\\
&=\langle x_1x_2 \otimes y,\Delta'_\prec(z)\rangle.
\end{align*}
\end{itemize}
So the induction hypothesis holds for $x$ in both cases. \end{proof}

\begin{cor} \label{37}
The restriction of the pairing $\langle-,-\rangle$ to $\h_{\SPF}$ is nondegenerate.
\end{cor}

\begin{proof} Let us assume it is degenerate. By lemma \ref{36}, its kernel $I$ is a non trivial dendriform biideal of $\h_{\SPF}^+$.
Any non-zero element of $I$ of minimal degree is then in $Prim_{tot}(\h_{\SPF}^+)$, as $I$ is a dendriform coideal.
By lemma \ref{35}, we obtain that $\tun \in I$: absurd, as $\langle \tun,\tun\rangle=1\neq 0$. \end{proof}

\section{Isometries between plane and special plane posets}

All the pairs of isomorphic Hopf algebras $\h_{\PP}$ and $\h_{\SPP}$, $\h_{\WNP}$ and $\h_{\SWNP}$, $\h_{\PF}$ and $\h_{\SPF}$
have Hopf pairings. The isomorphism between these Hopf algebras are not isometries: for example,
$\langle \tdeux,\tdeux\rangle=0$ whereas $\langle \tddeux{$1$}{$2$},\tddeux{$1$}{$2$}\rangle=1$.
 Our aim in this section is to answer the question if there is an isometric Hopf isomorphism between them. 
The answer is immediately negative for $\h_{\WNP}$ and $\h_{\SWNP}$, as the first one is nondegenerate whereas the second is degenerate.

\subsection{Isometric Hopf isomorphisms between free Hopf algebras}

\begin{prop} \label{38}
Let us assume that the characteristic of the base field is not $2$. Let $H$ and $H'$ be two graded, connected Hopf algebras, both with a homogeneous, symmetric, nondegenerate Hopf pairing, and both free. The following assertions are equivalent:
\begin{enumerate}
\item There exists a homogeneous, isometric Hopf algebra isomorphism between $H$ and $H'$.
\item  For all $n\geq 0$, the spaces $H_n$ and $H'_n$ are isometric.
\end{enumerate}
\end{prop}

\begin{proof} $1 \Longrightarrow 2$. Obvious.\\

$2 \Longrightarrow 1$. Let us fix for all $n\in \mathbb{N}^*$ a complement $V_n$ of $(H^{+2})_n$ in $H_n$, where $H^+$ is the augmentation ideal of $H$.
As $H$ is free, the direct sum $V$ of the $V_n$'s freely generates $H$. Moreover, any subspace of $V$ generates a free subalgebra of $H$.
In particular, the subalgebra $H_{\langle n \rangle}$ of $H$ generated by $V_1\oplus \ldots \oplus V_n$ is free. Moreover, it contains
$H_0\oplus \ldots \oplus H_n$, so for all $v \in V_0\oplus \ldots \oplus V_n$, $\Delta(v) \in H_{\langle n\rangle}\otimes H_{\langle n\rangle}$. So 
$H_{\langle n\rangle}$ is a Hopf subalgebra of $H$. Finally, it is the algebra generated by $H_0\oplus\ldots \oplus H_n$, so does not depend of the choice of $V$.
We similarly define $H'_{\langle n\rangle}$ for all $n$.\\

We are going to construct for all $n \geq 0$ a Hopf algebra isomorphism $\phi_n:H_{\langle n\rangle}\longrightarrow H'_{\langle n\rangle}$ such that:
\begin{enumerate}
\item $\phi_n$ is homogeneous of degree $0$.
\item For all $x,y \in H_{\langle n\rangle}$, $\langle \phi_n(x),\phi_n(y)\rangle=\langle x,y\rangle$.
\item $\phi_n$ restricted to $H_{\langle n-1 \rangle}$ is $\phi_{n-1}$ if $n \geq 1$.
\item For all $i\leq n$, $H'_i=(H'^{+2})_i \oplus \phi_n(V_i)$.
\end{enumerate}
As $H_{\langle 0\rangle}=H'_{\langle 0\rangle}=\K$, we define $\phi_0$ by $\phi_0(1)=1$. Let us assume that $\phi_{n-1}$ is defined.
Then $H_n=(H^{+2})_n\oplus V_n=(H_{\langle n-1\rangle})_n\oplus V_n$. By the induction hypothesis, $\phi_{n-1}$ induces an isometry
between $(H_{\langle n-1\rangle})_n$ and $(H'_{\langle n-1\rangle})_n=(H'^{+2})_n$. As $H_n$ and $H'_n$ are nondegenerate and isometric,
by Witt extension theorem, it can be extended into an isometry $\tilde{\phi}_{n-1}:H_n \longrightarrow H'_n$.
As $H_{\langle n \rangle}$ is freely generated by $V_0\oplus \ldots \oplus V_n$, we can define an algebra morphism
$\phi_n:H_{\langle n\rangle}\longrightarrow H'_{\langle n\rangle}$ by $\phi_n(v)=\phi_{n-1}(v)$ if $v\in V_i$, $i\leq n-1$
 and $\phi_n(v)=\tilde{\phi}_{n-1}(v)$ if $v \in V_n$. This algebra morphism immediately satisfies the points 3 and 4 of the induction, 
by construction of $\tilde{\phi}_{n-1}$, and also extends $\tilde{\phi}_{n-1}$. Moreover, by the fourth point, $\phi_n(V_1\oplus \ldots \oplus V_n)$ 
freely generated $H'_{\langle n\rangle}$, so $\phi_n$ is an algebra isomorphism from $H_{\langle n\rangle}$ to $H'_{\langle n\rangle}$.

Let us prove that $\phi_n$ is a Hopf algebra isomorphism. Let $x\in H_k$, $k\leq n$. For all $y\in H_i$, $z\in H_j$, $i+j=k$, as 
$\phi_n$ extends both $\phi_{n-1}$ and $\tilde{\phi}_{n-1}$, its restriction in all degree $\leq n$ is an isometry, so:
\begin{align*}
\langle \Delta\circ \phi_n(x), \phi_n(y) \otimes \phi_n(z)\rangle&=\langle \phi_n(x),\phi_n(y)\phi_n(z)\rangle\\
&=\langle \phi_n(x),\phi_n(yz)\rangle\\
&=\langle x,yz \rangle\\
&=\langle \Delta(x),y\otimes z\rangle\\
&=\langle (\phi_n \otimes \phi_n)\circ \Delta(x),\phi_n(y) \otimes \phi_n(z)\rangle.
\end{align*}
As $\phi_n$ is surjective in degree $\leq n$, and by homogeneity of the pairing of $H'$, we deduce that
$(\phi_n \otimes \phi_n)\circ \Delta(x)-\Delta\circ \phi_n(x)\in (H'\otimes H')^\perp=(0)$, as the pairing of $H'$ is nondegenerate.
As $H_1\oplus \ldots \oplus H_n$ generates $H_{\langle n\rangle}$, $\phi_n$ is a Hopf algebra morphism.

Finally, let us prove the second point of the induction. By homogeneity of the pairings of $H$ and $H'$, it is enough to prove it
for $x,y$ homogeneous of the same degree $k$. We proceed by induction on $k$. If $k\leq n$, we already noticed that $\phi_n$ is an isometry in degree $k$. 
Let us assume that the result is true at all rank $<k$, with $k>n$. As $(H_{\langle n \rangle})_k=((H_{\langle n \rangle})^{+2})_k$, 
we can assume that $x=x_1x_2$, with $x_1,x_2$ homogeneous of degree $<k$. Then, using the induction hypothesis on $x_1$ and $x_2$:
\begin{align*}
\langle \phi_n(x), \phi_n(y)\rangle&=\langle \phi_n(x_1)\phi_n(x_2),\phi_n(y)\rangle\\
&=\langle \phi_n(x_1) \otimes \phi_n(x_2),\Delta \circ \phi_n(y)\rangle\\
&=\langle \phi_n(x_1) \otimes \phi_n(x_2),(\phi_n \otimes \phi_n) \circ \Delta(y)\rangle\\
&=\langle x_1\otimes x_2,\Delta(y) \rangle\\
&=\langle x,y\rangle.
\end{align*}

\textit{Conclusion.} We define $\phi:H\longrightarrow H'$ by $\phi(x)=\phi_n(x)$ for all $x \in H_{\langle n\rangle}$. By the third point of the induction,
this does not depend of the choice of $n$. Then $\phi$ is clearly an isometric, homogeneous Hopf algebra isomorphism. \end{proof}

We can improve this result, in the following sense:

\begin{prop}
Let us assume that the characteristic of the base field is not $2$.  Let $H$ and $H'$ be two graded, connected Hopf algebras, both with a homogeneous, symmetric,
nondegenerate Hopf pairing, and both free. Let $V$ and $V'$ be subspaces of respectively $H$ and $H'$, $W$ and $W'$ graded subspaces of respectively $V$ 
and $V'$ generating Hopf subalgebras $h$ and $h'$ of $H$ and $H'$. We assume that $h$ is a non isotropic subspace of $H$.
The following assertions are equivalent:
\begin{enumerate}
\item There exists a homogeneous, isometric Hopf algebra isomorphism $\phi$ between $H$ and $H'$, such that $\phi(h)=h'$.
\item  For all $n\geq 0$, the spaces $H_n$ and $H'_n$ are isometric and the spaces $h_n$ and $h'_n$ are isometric.
\end{enumerate}\end{prop}

\begin{proof} $1 \Longrightarrow 2$. Obvious.\\

$2 \Longrightarrow 1$. For all $n \geq 1$, let us choose a complement $U_n$ of $W_n$ in $V_n$.

 By  proposition \ref{38}, there exists an isometric, homogeneous Hopf algebra isomorphism $\psi:h\longrightarrow h'$.
Let us construct inductively a Hopf algebra isomorphism $\phi_n:H_{\langle n\rangle}\longrightarrow H'_{\langle n\rangle}$,
isometric, such that:
\begin{enumerate}
\item $\phi_n$ is homogeneous of degree $0$.
\item For all $x,y \in H_{\langle n\rangle}$, $\langle \phi_n(x),\phi_n(y)\rangle=\langle x,y\rangle$.
\item $\phi_n$ restricted to $H_{\langle n-1 \rangle}$ is $\phi_{n-1}$ if $n \geq 1$.
\item $\phi_n(x)=\psi(x)$ for all $x\in h_{\langle n\rangle}$.
\item For all $i\leq n$, $H'_i=(H'^{+2})_i \oplus\psi(W_i)\oplus\phi_n(U_i)$.
\end{enumerate}
As $H_{\langle 0\rangle}=H'_{\langle 0\rangle}=K$, we define $\phi_0$ by $\phi_0(1)=1$. Let us assume that $\phi_{n-1}$ is defined.
Then $H_n=(H^{+2})_n\oplus W_n \oplus U_n=(H_{\langle n-1\rangle})_n\oplus W_n \oplus U_n$. By the induction hypothesis, $\phi_{n-1}$ and $\psi$
induces an isometry between $(H_{\langle n-1\rangle})_n\oplus W_n$ and $(H'_{\langle n-1\rangle})_n\oplus W'_n=(H'^{+2})_n\oplus W'_n$. 
As $H_n$ and $H'_n$ are nondegenerate and isometric, by the extension theorem of Witt, it can be extended into an isometry 
$\tilde{\phi}_{n-1}:H_n \longrightarrow H'_n$. As $H_{\langle n \rangle n}$ is freely generated by $V_0\oplus \ldots \oplus V_n$, 
we can define an algebra morphism $\phi_n:H_{\langle n\rangle}\longrightarrow H'_{\langle n\rangle}$ by $\phi_n(v)=\phi_{n-1}(v)$ if $v\in V_i$, $i\leq n-1$
and $\phi_n(v)=\tilde{\phi}_{n-1}(v)$ if $v \in V_n$. This morphisms clearly satisfy the fourth point of the definition.
The end of the proof is similar to the proof of proposition \ref{38}. \end{proof}

We shall apply these propositions with $H=\h_{\PP}$, $H'=\h_{\SPP}$, $V$ being the subspace generated by plane posets and $V'$
being the subspace generated by special plane posets, $W$ the subspace generated by plane trees and $W'$ the subspace generated
by special plane trees. We obtain the following results:

\begin{lemma}\label{40}\begin{enumerate}
\item The following assertions are equivalent:
\begin{enumerate}
\item There exists a homogeneous, isometric Hopf algebra isomorphism between $\h_{\PP}$ and $\h_{\SPP}$.
\item For all $n\geq 1$, $(\h_{\PP})_n$ and $(\h_{\SPP})_n$ are isometric.
\end{enumerate}
\item The following assertions are equivalent:
\begin{enumerate}
\item There exists a homogeneous, isometric Hopf algebra isomorphism $\phi$ between $\h_{\PF}$ and $\h_{\SPF}$.
\item For all $n\geq 1$, $(\h_{\PF})_n$ and $(\h_{\SPF})_n$ are isometric.
\end{enumerate}
\item The following assertions are equivalent:
\begin{enumerate}
\item There exists a homogeneous, isometric Hopf algebra isomorphism $\phi$ between $\h_{\PP}$ and $\h_{\SPP}$, such that $\phi(\h_{\SPF})=\h_{\SPF}$.
\item For all $n\geq 1$, $(\h_{\PP})_n$ and $(\h_{\SPP})_n$, $(\h_{\PF})_n$ and $(\h_{\SPF})_n$ are isometric.
\end{enumerate}\end{enumerate}\end{lemma}

In particular, if $\K$ is an algebraically closed field of characteristic $\neq 2$, 
two nondegenerate spaces are isometric, if, and only if, they have the same dimension. Hence, conditions (b) of Lemma \ref{40}
are all satisfied.

\begin{prop}
If $\K$ is an algebraically closed field of characteristic $\neq 2$, there exists a homogeneous, isometric Hopf algebra isomorphism 
$\phi$ between $\h_{\PP}$ and $\h_{\SPP}$, such that $\phi(\h_{\SPF})=\h_{\SPF}$.
\end{prop}

\subsection{Existence of an isometry between plane and special plane posets}

Let us precise the condition on the field for $\h_{\PP}$ and $\h_{\SPP}$ to be isometric:

\begin{theo}\label{44}
The following assertions are equivalent:
\begin{enumerate}
\item There exists a homogeneous, isometric Hopf algebra isomorphism between $\h_{\PP}$ and $\h_{\SPP}$.
\item The characteristic of the base field $\K$ is not $2$ and there exists $i \in \K$ such that $i^2=-1$.
\end{enumerate}
\end{theo}

\begin{proof} By lemma \ref{40}, the question is essentially to know if
$(\h_{\PP})_n$ and $(\h_{\SPP})_n$ are isometric.
More precisely, we are going to prove that the following assertions are equivalent:
\begin{enumerate}
\item For all $n\geq 1$, $(\h_{\PP})_n$ and $(\h_{\SPP})_n$ are isometric.
\item For all $n\geq 1$, $(\h_{\PP})_n$ and $(\h_{\SPP})_n$ have orthonormal bases.
\item The characteristic of the base field $\K$ is not $2$ and there exists $i \in \K$ such that $i^2=-1$.
\end{enumerate}
This will immediately imply theorem \ref{44}. Obviously, $2 \Longrightarrow 1$, 
as $(\h_{\PP})_n$ and $(\h_{\SPP})_n$ have the same dimension.\\

$1\Longrightarrow 3$. We choose $n=2$. In the basis $(\tdeux,\tun\tun)$ of $(\h_{\PP})_2=(\h_{\WNP})_2$, the matrix of the pairing is 
$\left(\begin{array}{cc} 0&1\\1&2\end{array}\right)$. In the basis $(\tddeux{$1$}{$2$},\tdun{$1$}\tdun{$2$})$ of $(\h_{\SPP})_2=(\h_{\SWNP})_2$, 
the matrix of the pairing is $\left(\begin{array}{cc} 1&1\\1&2\end{array}\right)$.
Considering the determinants of both matrices, we obtain that $1$ and $-1$ differ multiplicatively from a square of $\K$, so $-1$ is a square of $\K$.
For all $x=a\tdeux+b\tun\tun \in (\h_{\PP})_2$, $\langle x,x\rangle=2(ab+b^2)$. As $(\h_{\PP})_2$ is isometric with $(\h_{\SPP})_2$, there exists
$x \in (\h_{\PP})_2$, such that $\langle x,x\rangle=1$. As a consequence, $\Char(\K)\neq 2$.\\

$3\Longrightarrow 2$. 
As $H_{\SPP}$ is isometric to $\FQSym$, it is equivalent to prove that both $\h_{\PP}$ and $\FQSym$ have an orthonormal basis.
Let us fix $V=(H_{\SPP})_n$ or $(\FQSym)_n$ for a given $n$. Then $V$ has a basis $(e_i)_{i\in I}$, with the following properties:
there exists a partial order $\lll$ on $I$ and an involution $\iota:I\longrightarrow I$, such that for any $i,j\in I$,
\[\langle e_i,e_j\rangle\neq 0\Longrightarrow i\lll \iota(j).\]
Moreover, $\langle e_i,e_{\iota(i)}\rangle=1$. 
For $\FQSym$, any partial order $\lll$ on permutations is suitable, with $\iota(\sigma)=\sigma^{-1}$.
For $H_{\SPP}$, this is Lemma 35 of \cite{Foissy1}. Let us put $I'=\{i\in I,\:\iota(i)=i\}$ and $I''=I\setminus I'$. 
\begin{itemize}
\item Let $i,j\in I'$. If $\langle e_i,e_j\rangle \neq 0$, then $i\lll \iota(j)=j$; by symmetry, $\langle e_j,e_i\rangle \neq 0$,
so $j\lll\iota(i)=i$. As $\lll$ is an order, $i=j$. 
\item Let $i\in I'$ and $j\in I''$. If $\langle e_i, e_j \rangle \neq 0$, then $i\lll \iota(j)$.
By symmetry, $j\lll \iota(i)=i$, so $j\lll i\lll \iota(j)$.
\end{itemize}
Hence, considering a convenient total extension of $\lll$, in the basis $(e_i)_{i\in I}$ the matrix of the pairing  has the form
\[M=\begin{pmatrix}
*&*&A\\
*&I_l&0\\
A^T&0&0
\end{pmatrix},\]
where $A$ is antidiagonal, that is to say has the form:
\[A=\begin{pmatrix}
*&\ldots&*&1\\
\vdots&\iddots&\iddots&0\\
*&\iddots&\iddots&\vdots\\
1&0&\ldots&0
\end{pmatrix}.\]

\textit{First step}. Let us assume that $l=0$, that is to say $M$ is antidiagonal. Let us prove that 
there exists a basis $\mathcal{B}$ of $V$ such that the matrix of the pairing in this basis is
\[J_p=\begin{pmatrix}
0&\ldots&0&1\\
\vdots&\iddots&\iddots&0\\
0&\iddots&\iddots&\vdots\\
1&0&\ldots&0
\end{pmatrix}.\]
We proceed on the dimension $p$ of $V$. If $p=0$ or $1$, there is nothing to prove.
Otherwise, applying the result to $V'=Vect(e_2,\ldots,e_{p-1})$ (which is orthogonal to $e_p$), 
we can assume that
\[M_{2\leqslant i,j\leqslant p-1}=J_{p-2}.\]
For any $1\leq i\leq p$, let us put $e'_i=e_i-\lambda_i e_p$, with:
\[\lambda_i=\begin{cases}
\dfrac{1}{2}\langle e_1,e_1\rangle \mbox{ if }i=1,\\
\langle e_i,e_1\rangle \mbox{ if }2\leq i\leq p-1,\\
0\mbox{ if }i=p.
\end{cases}\]
Then $(e'_1,\ldots,e'_p)$ is a basis of $V$. As $\langle e_p,e_p \rangle=0$, for any $i,j$:
\[\langle e'_i,e'_j\rangle=\langle e_i,e_j\rangle-\lambda_i \langle e_i,e_p \rangle-\lambda_j \langle e_j,e_p \rangle.\]
Consequently:
\begin{itemize}
\item If $2\leq i,j\leq p-1$, $\langle e'_i,e'_j\rangle=\langle e_i,e_j\rangle$.
\item If $i=1$ and $1\leqslant j\leqslant p-1$, by choice of $\lambda_i$, $\langle e'_i,e'_j\rangle=0$.
\item If $1\leqslant i\leqslant p-1$ and $j=p$, then $\langle e'_i,e'_p\rangle=\langle e_i,e_p\rangle=\delta_{1,p}$.
\end{itemize}
So the matrix of the pairing is in this basis is $J_p$.\\

\textit{Second step.} We apply the first step to $Vect(e_i,i\in I'')$. Up to a change of basis of this subspace, we can assume that
\[M=\begin{pmatrix}
0&B&J_k\\
B^T&I_l&0\\
J_k&0&0
\end{pmatrix}\]
with $k,l\geq 0$ and $B\in M_{k,l}(\K)$. 
Let us consider the matrix
\[P=\begin{pmatrix}
I_k&0&0\\
0&I_l&0\\
0&-J_k^{-1}B&I_k
\end{pmatrix}.\]
This is invertible, and:
\[P^T M P=\begin{pmatrix}
0&0&J_k\\
0&I_l&0\\
J_k&0&0
\end{pmatrix}.\]
Hence, up to a permutation of the vectors of the basis formed by the column of $P$,
there exists a basis $(e'_1,\ldots,e'_p)$ of $V$, such that the matrix of the pairing in this basis is
diagonal by blocks, with diagonal blocks equal to $(1)$ or $\left(\begin{array}{cc}0&1\\1&0\end{array}\right)$.
 Now, observe that, denoting by $i$ one of the square root of $-1$ in $\K$:
\[\begin{pmatrix}
\frac{i}{2}&-i\\\frac{1}{2}&1 
\end{pmatrix}\begin{pmatrix}
0&1\\1&0\
\end{pmatrix}
\begin{pmatrix}
\frac{i}{2}&\frac{1}{2}\\-i&1
\end{pmatrix}
=\begin{pmatrix}
1&0\\0&1
\end{pmatrix}.\]
So $V$ has an orthogonal basis. \\

As a conclusion, $(\h_{\PP})_n$ and $(\h_{\SPP})_n$ have an orthogonal basis. \end{proof}

\begin{remark}
The same proof can be applied to $\h_{\PF}$ and $\h_{\WNP}$: 
if Condition 2 of Theorem \ref{44} is satisfied, then for any $n\geqslant 1$,
$(\h_{\PF})_n$  and $(\h_{\WNP})_n$ have orthonormal bases. We conjecture that if Condition 2 of Theorem \ref{44}
is satisfied, then $\h_{\SPF}$ has also an orthonormal basis, giving Condition 3.(b) of Lemma \ref{40}.
\end{remark}

\begin{example}
Let $i$ be one of the two square roots of $-1$ in $\K$. 
We define an isometry from $(\h_{\PP})_{\langle 2\rangle}$ to $(\h_{\SPP})_{\langle 2\rangle}$ by:
\[\left\{\begin{array}{rcl}
\phi(\tun)&=&\tdun{$1$},\\[2mm]
\phi(\tdeux)&=&\displaystyle i\tddeux{$1$}{$2$}+\frac{1+i}{2}\tdun{$1$}\tdun{$2$}.
\end{array}\right.\]
Using direct computations, it is possible to extend $\phi$ from $(\h_{\PP})_{\langle 3\rangle}$ to $(\h_{\SPP})_{\langle 3\rangle}$
sending $(\h_{\WNP})_{\langle 3\rangle}$ to $(\h_{\SWNP})_{\langle 3\rangle}$ in four families of isometries parametrized by an element $x\in \K$ by:
\begin{enumerate}
\item \[\left\{\begin{array}{rcl}
\phi_1(\ttroisdeux)&=&\displaystyle \tdtroisdeux{$1$}{$2$}{$3$}+(ix-i)\tddeux{$1$}{$2$}\tdun{$3$}+(-1-ix)\tdun{$1$}\tddeux{$2$}{$3$}
+\frac{1+i}{2}\tdun{$1$}\tdun{$2$}\tdun{$3$},\\[4mm]
\phi_1(\ttroisun)&=&\displaystyle (-1-i+3x)\tdtroisdeux{$1$}{$2$}{$3$}-i\tdtroisun{$1$}{$3$}{$2$}
+\frac{3ix^2-2ix}{2}\tddeux{$1$}{$2$}\tdun{$3$}\\
&&\displaystyle +\frac{-3ix^2+(-3+i)x+2+i}{2}\tdun{$1$}\tddeux{$2$}{$3$}+x\tdun{$1$}\tdun{$2$}\tdun{$3$},\\[4mm]
\phi_1(\ptroisun)&=&\displaystyle (-3x+2+2i)\tdtroisdeux{$1$}{$2$}{$3$}-i\pdtroisun{$3$}{$1$}{$2$}
+\frac{3ix^2-2ix}{2}\tddeux{$1$}{$2$}\tdun{$3$}\\
&&\displaystyle +\frac{3ix^2+(6-4i)x-4-2i}{2}\tdun{$1$}\tddeux{$2$}{$3$}+(-x+1+i)\tdun{$1$}\tdun{$2$}\tdun{$3$}.
\end{array}\right.\]
\item \[\left\{\begin{array}{rcl}
\phi_2(\ttroisdeux)&=&\displaystyle \tdtroisdeux{$1$}{$2$}{$3$}+(ix-i)\tddeux{$1$}{$2$}\tdun{$3$}+(-1-ix)\tdun{$1$}\tddeux{$2$}{$3$}
+\frac{1+i}{2}\tdun{$1$}\tdun{$2$}\tdun{$3$},\\[4mm]
\phi_2(\ttroisun)&=&\displaystyle (-1-i+3x)\tdtroisdeux{$1$}{$2$}{$3$}-i\tdtroisun{$1$}{$3$}{$2$}
+\frac{3ix^2-2ix}{2}\tddeux{$1$}{$2$}\tdun{$3$}\\
&&\displaystyle +\frac{-3ix^2+(-3+i)x+2+i}{2}\tdun{$1$}\tddeux{$2$}{$3$}+x\tdun{$1$}\tdun{$2$}\tdun{$3$},\\[4mm]
\phi_2(\ptroisun)&=&\displaystyle (-3x+2)\tdtroisdeux{$1$}{$2$}{$3$}+2i\tdtroisun{$1$}{$3$}{$2$}+i\pdtroisun{$3$}{$1$}{$2$}
+\frac{-3ix^2+4ix-6i}{2}\tddeux{$1$}{$2$}\tdun{$3$}\\
&&\displaystyle +\frac{3ix^2+(6-4i)x-4-2i}{2}\tdun{$1$}\tddeux{$2$}{$3$}+(-x+1+i)\tdun{$1$}\tdun{$2$}\tdun{$3$}.
\end{array}\right.\]
\item If the characteristic of the base field is not 2, nor 3:
\[\left\{\begin{array}{rcl}
\phi_3(\ttroisdeux)&=&\displaystyle -\tdtroisdeux{$1$}{$2$}{$3$}+\frac{-3ix-i}{3}\tddeux{$1$}{$2$}\tdun{$3$}
+\frac{3ix-2i+3}{3}\tdun{$1$}\tddeux{$2$}{$3$}
+\frac{3i-1}{6}\tdun{$1$}\tdun{$2$}\tdun{$3$},\\[4mm]
\phi_3(\ttroisun)&=&\displaystyle (-1-i+3x)\tdtroisdeux{$1$}{$2$}{$3$}-i\tdtroisun{$1$}{$3$}{$2$}
+\frac{3ix^2-2ix}{2}\tddeux{$1$}{$2$}\tdun{$3$}\\
&&\displaystyle +\frac{-3ix^2+(-3+i)x+2+i}{2}\tdun{$1$}\tddeux{$2$}{$3$}+x\tdun{$1$}\tdun{$2$}\tdun{$3$},\\[4mm]
\phi_3(\ptroisun)&=&\displaystyle (-3x+2i)\tdtroisdeux{$1$}{$2$}{$3$}-i\pdtroisun{$3$}{$1$}{$2$}
+\frac{-9ix^2-2i}{6}\tddeux{$1$}{$2$}\tdun{$3$}\\
&&\displaystyle +\frac{9ix^2+18x-10i}{6}\tdun{$1$}\tddeux{$2$}{$3$}+\frac{-3x+3i+1}{3}\tdun{$1$}\tdun{$2$}\tdun{$3$}.
\end{array}\right.\]
\item If the characteristic of the base field is neither 2, nor 3:
\[\left\{\begin{array}{rcl}
\phi_4(\ttroisdeux)&=&\displaystyle -\tdtroisdeux{$1$}{$2$}{$3$}+\frac{-3ix-i}{3}\tddeux{$1$}{$2$}\tdun{$3$}
+\frac{3ix-2i+3}{3}\tdun{$1$}\tddeux{$2$}{$3$}
+\frac{3i-1}{6}\tdun{$1$}\tdun{$2$}\tdun{$3$},\\[4mm]
\phi_4(\ttroisun)&=&\displaystyle (-1-i+3x)\tdtroisdeux{$1$}{$2$}{$3$}-i\tdtroisun{$1$}{$3$}{$2$}
+\frac{3ix^2-2ix}{2}\tddeux{$1$}{$2$}\tdun{$3$}\\
&&\displaystyle +\frac{-3ix^2+(-3+i)x+2+i}{2}\tdun{$1$}\tddeux{$2$}{$3$}+x\tdun{$1$}\tdun{$2$}\tdun{$3$},\\[4mm]
\phi_4(\ptroisun)&=&\displaystyle -3x\tdtroisdeux{$1$}{$2$}{$3$}+2i\tdtroisun{$1$}{$3$}{$2$}+i\pdtroisun{$3$}{$1$}{$2$}
+\frac{-9ix^2-14i}{6}\tddeux{$1$}{$2$}\tdun{$3$}\\
&&\displaystyle +\frac{9ix^2+18x-10i}{6}\tdun{$1$}\tddeux{$2$}{$3$}+\frac{-3x+3i+1}{3}\tdun{$1$}\tdun{$2$}\tdun{$3$}.
\end{array}\right.\]
\end{enumerate}\end{example}

\section{Conclusion}

We finally obtain the following commuting diagram:
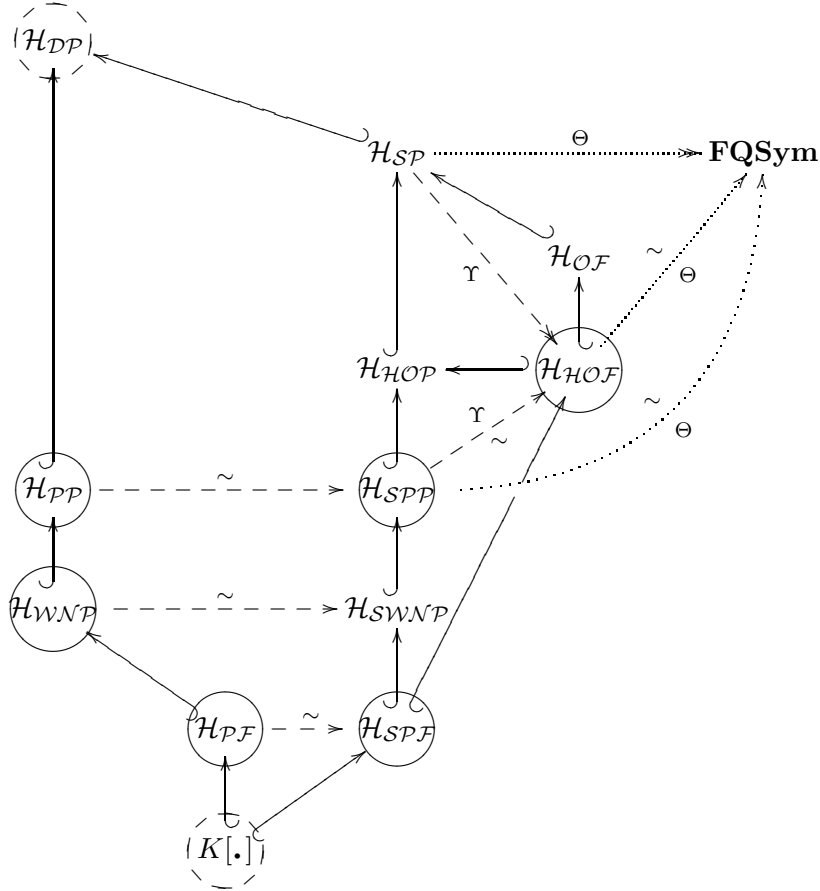
\[\xymatrix{\rondrond{\h_{\DP}}&&&&\\
&&\h_{\SP} \ar@{_{(}->}[llu]\ar@{-->>}_{\Upsilon}[ddr]\ar@{.>>}^{\Theta}[rr]
&&\FQSym\\&&&\h_{\OF} \ar@{_{(}->}[lu]&\\
&&\h_{\HOP} \ar@{^{(}->}[uu]&
\rond{\h_{\HOF}} \ar@{_{(}->}[u]\ar@{_{(}->}[l]\ar@{.>}^{\sim}_{\Theta}[ruu]&\\
\rond{\h_{\PP}} \ar@{-->}[rr]^{\sim}\ar@{^{(}->}[uuuu]&&
\rond{\h_{\SPP}} \ar@{-->}[ur]^{\Upsilon}_{\sim}
\ar@{^{(}->}[u]\ar@/_4pc/@{.>}^{\sim}_{\Theta}[rruuu]&&\\
\rond{\h_{\WNP}} \ar@{-->}[rr]^{\sim}\ar@{^{(}->}[u]&&
\h_{\SWNP} \ar@{^{(}->}[u]&&\\
&\rond{\h_{\PF}} \ar@{-->}[r]^{\sim}\ar@{^{(}->}[lu]&
\rond{\h_{\SPF}} \ar@{^{(}->}[u]\ar@{_{(}->}|!{[uu];[ruu]}\hole[ruuu]&&\\
&\rondrond{\K[\tun]}\ar@{_{(}->}[u]\ar@{_{(}->}[ru]&&&}\]

On the first column, algebras stable under $\prodh$ and $\iota$ (see definitions in \cite{Foissy1}). On the third and fourth columns, algebras stable under $\nwarrow$,
$\Delta_\prec$ and $\Delta_\succ$. The algebras such that the restriction of the pairing $\langle-,-\rangle$ is nondegenerate are circled. If the circle is dotted, 
the result is true if, and only if, the characteristic of the base field is zero.
The three horizontal dotted lines correspond to the isomorphisms sending $(P,\leq_h,\leq_r)$ to $(P,\leq_h,\leq)$.
Moreover, it is not difficult to show that the intersection of two Hopf algebras of this diagram is given by the smallest common ancestor 
in the oriented graph formed by the black edges of this diagram. This lies on the fact the only plane posets $(P,\leq_h,\leq_r)$ which are special
(recall that this means that $\leq_r$ is total) are the double posets $\tun^n$, for all $n\geq 0$.\\

All the arrows of the diagram are isometries, at the exception of the three horizontal dotted lines. There exists isometric Hopf algebra isomorphisms
between $\h_{\PP}$ and $\h_{\SPP}$, if, and only if, the characteristic of the base field $\K$ is not $2$ and $-1$ is a square of $\K$.\\

If the characteristic of $\K$ is zero, all these Hopf algebras are free, cofree, and self-dual.

\bibliographystyle{amsplain}
\addcontentsline{toc}{section}{References}
\bibliography{biblio}

\end{document}